\documentclass[final,leqno]{siamltex}

\usepackage{amsfonts,amssymb,amsmath,cmmib57,color,graphics,mathdots}
\usepackage{amstext}
\usepackage{amscd}
\usepackage{graphicx}
\usepackage{longtable}
\usepackage{calc}
\usepackage[latin1]{inputenc}
\usepackage{url}
\usepackage{diagbox}
\usepackage{hyperref}

\newtheorem{remark}[theorem]{Remark}
\newtheorem{example}{Example}

\newsavebox{\savepar}

\newcommand{\NN}{\mathbb{N}}
\newcommand{\RR}{\mathbb{R}}
\newcommand{\CC}{\mathbb{C}}

\newcommand{\Hb}{\mathrm{HKCF}}

\newcommand{\cK}{{\cal K}}
\newcommand{\cA}{{\cal A}}

\newcommand{\cD}{{\cal D}}
\newcommand{\cM}{{\cal M}}
\newcommand{\cS}{{\cal S}}
\newcommand{\cH}{{\cal H}}

\newcommand{\cP}{{\cal P}}
\newcommand{\cR}{{\cal R}}

\newcommand{\cL}{{\cal L}}
\newcommand{\cJ}{{\cal J}}

\newcommand{\la}{\lambda}

\newcommand{\orb}{{\cal O}}
\newcommand{\bun}{{\cal B}}
\newcommand {\mat}  [1] {\left[\begin{array}{#1}}
\newcommand {\rix}      {\end{array}\right]}

\DeclareMathOperator{\tr}{trace}

\DeclareMathOperator{\pen}{PENCIL}

\DeclareMathOperator{\codim}{codim}
\DeclareMathOperator{\syst}{syst}

\newcounter{algo}[section]

\DeclareMathOperator{\im}{im}

\newcommand{\C}{\mathbb{C}}

\title{Generic eigenstructures of Hermitian pencils}

\author{Fernando De Ter\'{a}n\thanks{Departamento de Matem\'aticas, Universidad Carlos III de Madrid, Avda. Universidad 30, 28911 Legan\'es, Spain. {\tt fteran@math.uc3m.es}} \and Andrii Dmytryshyn\thanks{School of Science and Technology, \"{O}rebro University, 701 82, \"{O}rebro, Sweden. {\tt andrii.dmytryshyn@oru.se}}\and Froil\'an M. Dopico\thanks{Departamento de Matem\'aticas, Universidad Carlos III de Madrid, Avda. Universidad 30, 28911 Legan\'es, Spain. {\tt dopico@math.uc3m.es}}}

\usepackage[notcite,notref]{showkeys}

\begin{document}

\date{\today}

\maketitle

\begin{abstract}
  We obtain the generic complete eigenstructures of complex Hermitian $n\times n$ matrix pencils with rank at most $r$ (with $r\leq n$). To do this, we prove that the set of such pencils is the union of a finite number of bundle closures, where each bundle is the set of complex Hermitian $n\times n$ pencils with the same complete eigenstructure (up to the specific values of the finite eigenvalues). We also obtain the explicit number of such bundles and their codimension. The cases $r=n$, corresponding to general Hermitian pencils, and $r<n$ exhibit surprising differences, since for $r<n$ the generic complete eigenstructures can contain only real eigenvalues, while for $r=n$ they can contain real and non-real eigenvalues. Moreover, we will see that the sign characteristic of the real eigenvalues plays a relevant role for determining the generic eigenstructures of Hermitian pencils.
\end{abstract}






\section{Introduction}

The {\em complete eigenstructure} of a matrix pencil (or just a ``pencil", for short) is an intrinsic information of the pencil that is relevant in many of the applied problems where matrix pencils (or, more in general, matrix polynomials and rational matrices) arise (see, for instance,  \cite{DeKa87,DeKa88,DeKa93,MaMT15,Van79,Van81} and the references therein). More precisely, the complete eigenstructure of a pencil is the set of invariants of the pencil under {\em strict equivalence}, and it is encoded in the {\em Kronecker canonical form} (see \cite[Ch. XII]{Gant59} or \cite[\S 3]{LaRo05} for a more recent reference). In many applications where matrix pencils arise (either by themselves or by means of {\em linearizations} of matrix polynomials and rational matrices) the coefficient matrices have some particular symmetries, which lead to {\em structured matrix pencils}. These include {\em (skew-)symmetric}, {\em (skew-)Hermitian}, {\em (anti-)palindromic}, or {\em alternating} matrix pencils (see, for instance, \cite{dopico2021strongly,MMMM06b,MaMT15}).

The problem addressed in the present work is an instance of the general problem of determining the {\em most likely} complete eigenstructure of matrix pencils within some particular set, $\cS$. We use the word {\em generic} for the most likely complete eigenstructure, to mean that all pencils within the set $\cS$ are in the closure of the set of pencils having the generic eigenstructure. In other words, in every neighborhood of any particular pencil in $\cS$ there is at least one pencil having the generic eigenstructure. The generic complete eigenstructure of general $n\times n$ pencils consist, as it is well-known, of $n$ different simple eigenvalues. However, when some restrictions are imposed to the pencils, so that we restrict ourselves to pencils in some particular set $\bf\cS$, then it is not, in general, so easy to identify the generic complete eigenstructure (for instance, it is not trivial to obtain the generic complete eigenstructure of general $m\times n$ pencils when $m\neq n$, see below). As a consequence, the problem of describing the generic complete eigenstructure of matrix pencils within a particular set has attracted the attention of researchers for several decades. The research on this problem has allowed to describe the generic eigenstructure for the following sets of matrix pencils:
\begin{itemize}
    \item Singular $n\times n$ (namely, square) pencils \cite{Wate84}.
    \item Rectangular pencils of a fixed size \cite{DeEd95} (though the credit of the result, as mentioned in \cite[p. 85]{DeEd95}, goes back to, at least, \cite{VanD79}).
    \item General $m\times n$ pencils with rank at most $r$ (smaller than $\min\{m,n\}$) \cite{DeDo08} (revisited in \cite{DeDL17}).
    \item Palindromic and alternating $n\times n$ pencils with rank at most $r$ (smaller than $n$) \cite{DeTe17}.
    \item Complex symmetric $n\times n$ pencils with rank at most $r$ (smaller than $n$) \cite{ddd20}.
    \item Complex skew-symmetric $n\times n$ pencils with rank at most $2r$ (smaller than $n$) \cite{DmDo18}.
\end{itemize}
If we remove the restriction for the pencils being of bounded rank, then the generic complete eigenstructure of structured $n\times n$ pencils is also known for the following structures (though, up to our knowledge, some of them are not explicitly provided in the literature):
\begin{itemize}
    \item Complex symmetric $n\times n$ pencils: it is the same as for general (non-structured) matrix pencils, namely $n$ different eigenvalues. A way to see this is the following. Consider the set of $n\times n$ symmetric pencils as a manifold depending on $n(n+1)$  complex variables, encoded in a vector $X$ (these variables come from the upper triangular part, including the main diagonal, of the leading and the trailing coefficient matrices of the pencil). Assume that $f(X,\lambda)=\det\cP(X,\lambda)=\sum_{i=0}^n p_i(X)\lambda^{i}$ is the determinant of a general $n\times n$ symmetric pencil, $\cP(X,\la)$. Then, the subset of singular symmetric pencils is a proper algebraic set of $\CC^{n(n+1)}$, defined by the polynomial equations $p_0(X)=\cdots=p_n(X)=0$. Second, the subset of pencils with a multiple eigenvalue is also an algebraic set, namely Res($f(X,\lambda),f'(X,\lambda))=0$, where $f'(X,\lambda)$ is the derivative of $f(X,\la)$ with respect to the variable $\la$, and Res denotes the resultant (which is a polynomial in $X$) \cite[Ch. I, \S3]{macaulay}. Therefore, the set of $n\times n$ symmetric pencils with $n$ different eigenvalues is the complementary of the union of two algebraic sets. Since it is also nonempty (for instance, $\diag(\la-1,\la-2,\hdots,\la-n)$ is such a pencil), we conclude that it is a generic set.
    \item Complex skew-symmetric $n\times n$ pencils: we have been unable to find an explicit expression for the generic complete eigenstructures for this structure. However, it can be deduced from the canonical form under congruence of skew-symmetric pencils in \cite{Thom91} and the developments in \cite{DmDo18}. More precisely, it depends on whether $n$ is even or odd. If $n$ is even, then the generic canonical form consists of $n/2$ skew-symmetric Jordan blocks associated with different eigenvalues of multiplicity $2$, whereas if $n$ is odd we add to the previous structure a $0$ diagonal block. In other words, the generic complete eigenstructure consists of $n/2$ distinct eigenvalues of multiplicity exactly $2$ (if $n$ is even), together with one left and one right minimal indices equal to $0$ (when $n$ is odd).
    \item $\top$-palindromic pencils: the generic complete eigenstructure is also different depending on whether $n$ is even or odd \cite[Th. 6]{DeDo11}. More precisely, it consists of $n/2$ pairs of different simple complex values of the form $(\mu,1/\mu)$ (if $n$ is even), together with a simple eigenvalue $-1$ (when $n$ is odd). For $\top$-anti-palindromic pencils the eigenvalue $-1$ is replaced by $1$.  For $*$-palindromic pencils the generic complete eigenstructure can be found in \cite[Th. 5.4]{de2011equation} and it consists of $n/2$ pairs of different simple complex values of the form $(\mu,1/\overline{\mu})$ with $|\mu|>1$ (if $n$ is even), together with a simple eigenvalue which is an unspecified complex number of modulus $1$ (when $n$ is odd).
    \item $\top$-alternating pencils: in this case, the generic complete eigenstructure (that can be obtained from the one of $\top$-palindromic pencils by means of a Cayley transformation, see \cite{DeTe17}) consists of $n/2$ pairs of different simple complex values of the form $(\mu,-\mu)$, if $n$ is even, together with a simple eigenvalue $\infty$ (for $\top$-even pencils) or $0$ (for $\top$-odd pencils), when $n$ is odd. For $*$-alternating pencils, the pairs $(\mu,-\mu)$ are replaced by $(\mu,-\overline\mu)$.
\end{itemize}
In the present work, we describe the generic complete eigenstructure of Hermitian $n\times n$ matrix pencils with rank at most $r$, for any $0\leq r\leq n$. When $r=n$, what we obtain is the generic complete eigenstructure of general $n\times n$ Hermitian pencils (without restrictions), and for this reason this case is addressed separately. We will prove (in Theorems \ref{regular_th} and \ref{main_th}) that the number of generic complete eigenstructures in the set of $n\times n$ Hermitian pencils with rank at most $r$ is equal to $\left(\left\lfloor\frac{r}{2}\right\rfloor+1\right)\left\lfloor\frac{r+3}{2}\right\rfloor$. However, there are relevant differences between the case $r=n$ and $r<n$, namely:
\begin{itemize}
    \item When $r=n$, all generic eigenstructures correspond to regular pencils having $n$ simple eigenvalues. Some of these eigenvalues are real, and the other ones are pairs of non-real complex conjugate numbers. Only in one of these eigenstructures (namely, when all eigenvalues are real) there are no non-real eigenvalues.
    \item When $r < n$, however, none of the generic eigenstructures have non-real eigenvalues.
\end{itemize}
In both cases, each of the generic eigenstructures differs from the others in the number of real eigenvalues, and in the sign characteristics of these eigenvalues. This emphasizes the relevance of the sign characteristic for Hermitian pencils, which is a quantity that does not arise in the other structures mentioned above.

It is worth to emphasize also that the number of generic complete eigenstructures of $n\times n$ Hermitian pencils with rank at most $r$ is always greater than $1$, so there is no a unique generic eigenstructure. In fact, there are many when $r$ is large. This should not be surprising when $r<n$ since it happens also for general square singular pencils \cite{DeEd95} and for singular symmetric pencils \cite{ddd20}.  However, it may seem surprising in the case of general $n\times n$ Hermitian pencils (namely, when $r=n$). We will see that the lack of uniqueness when $r=n$ is a consequence of the different
forms to distribute the eigenvalues between real and non-real ones and also of the different possible sign characteristics of the real eigenvalues.

The rest of the paper is organized as follows: in Section \ref{notation_sec} we introduce the notation and some basic notions that are used throughout the manuscript. Section \ref{technical_sec} presents some technical results that are needed to prove the main results of the paper. The main results are introduced in Sections \ref{regular_sec} and \ref{main_sec}. More precisely, Theorem \ref{regular_th} describes the generic complete eigenstructures of Hermitian $n\times n$ matrix pencils, and in Theorem \ref{main_th} we provide the generic complete eigenstructures of Hermitian $n\times n$ matrix pencils with rank at most $r$, with $r<n$. The codimension of these generic complete eigenstructures are computed in Section \ref{codim_sec}, whereas in Section \ref{experiments_sec} we provide some numerical experiments to show that all the generic complete eigenstructures of general Hermitian pencils arise in numerical computations, and that non-real eigenvalues do not typically appear in singular Hermitian pencils. Finally, in Section \ref{conclusion_sec} we present a summary of the contributions of the manuscript.

\section{Basic definitions and notation}\label{notation_sec}

By $\RR$ and $\CC$ we denote the fields of real and complex numbers, respectively.  We also follow the standard notation ${\rm re}(\mu)$ and ${\rm im}(\mu)$ for, respectively, the real and imaginary parts of the complex number $\mu$, and $\rm i$ for the imaginary unit (namely, ${\rm i}=\sqrt{-1}$).

A {\em matrix pencil} is of the form ${\cal P}(\la)=A+\la B$, with $A,B\in\CC^{m\times n}$, and $\la$ being a scalar variable (matrix pencils can also be seen as pairs of $m\times n$ complex matrices $(A,B)$, see, for instance, \cite{LaRo05}). We use calligraphic letters, as above, to denote matrix pencils. Sometimes, and for the sake of brevity, we will remove the variable $\la$ and just write $\cP$. The pencil $\cP(\la)$ is called {\em regular} if $m=n$ and $\det\cP(\la)$ is not identically zero (as a polynomial in $\lambda$) and it is called {\em singular} otherwise. For a matrix pencil $\cP(\la)$ as above, we set $\cP(\la)^*=(A+\la B)^*=A^*+\la B^*$,  where $*$ denotes the conjugate transpose. It is important to note that the complex conjugation does not affect the variable $\la$.

In this paper, we are interested in complex {\em Hermitian} matrix pencils, namely those with $A^*=A$ and $B^*=B$. An important part of this work focuses on Hermitian matrix pencils with bounded rank, where the {\em rank} of the pencil $\cP$, denoted $\rank \cP$, is the size of the largest non-identically zero minor of $\cP$ (namely, the rank of $\cP$ when viewed as a matrix with entries in the field of rational functions in $\la$). The set of complex Hermitian $n\times n$ pencils is denoted by $\pen_{n\times n}^H$, and $\pen_{n\times n}^H(r)$ denotes the set of complex Hermitian $n\times n$ pencils with rank at most $r$.

The {\em signature} of a Hermitian constant matrix $A \in\CC^{n\times n}$ is the tuple $(\sigma_+,\sigma_-,\sigma_0)$, where $\sigma_+$ is the number of positive eigenvalues, $\sigma_-$ is the number of negative eigenvalues, and $\sigma_0$ is the multiplicity of the $0$ eigenvalue of $A$.

 Two $n\times n$ pencils $\cH_1 (\la)$ and $\cH_2 (\la)$ are $*$-congruent if there exists a nonsingular matrix $Q \in\CC^{n\times n}$ such that  $\cH_1 (\la) = Q^* \, \cH_2 (\la) \, Q$. Note that, if $\cH_1$ and $\cH_2$ are $*$-congruent, then $\cH_1$ is Hermitian if and only if $\cH_2$ is Hermitian. Since in the rest of this paper we only use the relation of ``$*$-congruence'', we will often refer to it simply as ``congruence''.

The {\em closure} of a subset of $n\times n$ complex matrix pencils ${\cal S}$, denoted by $\overline{\cal S},$ is considered in the Euclidean topology of the space $\CC^{2n^2}\simeq\CC^{n\times n}\times\CC^{n\times n}$, which is identified with the set of $n\times n$ matrix pencils, when considered as pairs of $n\times n$ matrices.  Also, open sets and open neighborhoods, as well as the notion of convergence, are considered in this topology. Through the identification above, $\pen_{n\times n}^H(r)$ becomes a subset of $\CC^{2n^2}$ and we can consider in $\pen_{n\times n}^H(r)$ the subspace topology induced by the Euclidean topology of $\CC^{2n^2}$.

The {\em direct sum} of the pencils $\cP_1,\hdots,\cP_k$ is a block diagonal pencil whose diagonal blocks are $\cP_1,\hdots,\cP_k$, in this order. We will denote it by either $\diag(\cP_1,\hdots,\cP_k)$ or $\bigoplus_{i=1}^k\cP_i$.

Following \cite[p. 909]{ddd20}, let ${\cal L}_d(\la):=\la G_d+F_d$,  where
\begin{equation*}
F_d :=
\begin{bmatrix}
0&1&&\\
&\ddots&\ddots&\\
&&0&1\\
\end{bmatrix}_{d\times(d+1)} \qquad\mbox{\rm and}\qquad
G_d :=
\begin{bmatrix}
1&0&&\\
&\ddots&\ddots&\\
&&1&0\\
\end{bmatrix}_{d\times(d+1)},
\end{equation*}
and define the Hermitian (actually, real symmetric) pencil
$$
{\cal M}_d(\la):=
\begin{bmatrix}0&{\cal L}_d(\la)^\top\\
{\cal L}_d(\la)&0
\end{bmatrix}_{(2d+1)\times(2d+1)}.
$$
The pencil ${\cal M}_0$ is a $1\times 1$ null matrix, and it is a degenerate case of ${\cal M}_d$ obtained after joining ${\cal L}_0$ and ${\cal L}_0^\top$, namely a null column and a null row, respectively.

We are also going to use the following pencils:

\begin{itemize}
\item Symmetric Jordan-like blocks associated with a finite eigenvalue:
$$
{\cal J}_k^H(\mu):=\left[\begin{array}{cccc}&&1&\la-\mu\\&\iddots&\iddots&\\1&\la-\mu&&\\\la-\mu&&&\end{array}\right]_{k\times k}\quad (\mu\in\CC).
$$
\item Hermitian Jordan-like blocks associated with the infinite eigenvalue:
$$
{\cal J}_k^H(\infty):=\left[\begin{array}{cccc}&&\la&1\\&\iddots&\iddots&\\\la&1&&\\1&&&\end{array}\right]_{k\times k}.
$$
\item Hermitian Jordan-like blocks associated with a pair of complex conjugate eigenvalues:
$$
{\cal J}_k^H(\mu,\overline\mu)=\begin{bmatrix}
0&{\cal J}_k^H(\overline\mu)\\{\cal J}_k^H(\mu)&0
\end{bmatrix}_{2k\times2k}.
$$
\end{itemize}
Note that
$$
\begin{array}{lc}
{\cal J}_1^H(\mu)=\cJ_1(\mu)=\la-\mu,&\mbox{for $\mu\in\CC$, and}\\
\cJ_1^H(\infty)=\cJ_1(\infty)=1.
\end{array}
$$
The last two blocks are standard Jordan $1\times 1$ blocks, and for this reason we will just write $\cJ_1(\mu)$ and $\cJ_1(\infty)$, respectively, omitting the superscript $H$. Note also that the Jordan-like block ${\cal J}_k^H(\mu)$ is Hermitian if and only if $\mu\in\RR$.

We also warn the reader that $\cJ_k^H(\mu,\overline{\mu})$ has size $2k\times2k$, instead of $k\times k$.

The following result, which provides a canonical form for Hermitian pencils under $*$-congruence, can be found in \cite[Th. 6.1]{LaRo05}, but here we present it as in \cite[Th. 1]{dmm22}.

\begin{theorem}\label{HCKF_th} {\rm(Hermitian Kronecker canonical form).}
Every $n\times n$ Hermitian matrix pencil, $\cH(\la)$, is $*$-congruent to a direct sum of blocks of the form
\begin{itemize}
\item[{\rm (i)}] blocks $\sigma \cJ_k^H(a)$, with $a\in\mathbb R$ and $\sigma\in\{+1,-1\}$;
\item[{\rm (ii)}] blocks $\sigma\cJ_k^H(\infty)$, with $\sigma\in\{+1,-1\}$;
\item[{\rm (iii)}] blocks $\cJ_k^H(\mu,\overline\mu)$, with $\mu\in\mathbb C$ having positive
imaginary part;
\item[{\rm (iv)}] blocks $\cM_k(\la)$.
\end{itemize}
The parameters $a,k,\sigma$, and $\mu$ may be
distinct in different blocks. These parameters, as well as the number of blocks of each type, are uniquely determined by $\cH$, and they are the invariants of $\cH$ under $*$-congruence. Furthermore, the direct sum is unique up to permutation of blocks. We will refer to this direct sum as the {\rm Hermitian Kronecker canonical form} of $\cH$, and we denote it by {\rm HKCF}$(\cH)$.
\end{theorem}

The values $a$ associated with the blocks in part (i) of Theorem \ref{HCKF_th} are the {\em real eigenvalues} of $\cH$, whereas the values $\mu,\overline{\mu}$ associated with blocks in part (iii) are the {\em pairs of (non-real) complex conjugate eigenvalues} of $\cH$. Both the real and the complex conjugate eigenvalues conform the set of {\em finite eigenvalues} of $\cH$. Moreover, if at least one block like the ones in part (ii) appears in HKCF($\cH$), then $\cH$ has an {\em infinite} eigenvalue. The list of signs $\sigma$ appearing in the blocks $\sigma\cJ_k^H(a)$ and $\sigma\cJ_k^H(\infty)$, given in a certain order, is known as the {\em sign characteristic} of the pencil $\cH$ \cite{LaRo05}. We emphasize that the sign characteristic of Hermitian matrix pencils and polynomials has been defined in several equivalent ways in the literature (see, for instance, \cite{lancaster2013sign,lancaster2021spectral,mntx20}). Each block $\cM_k$ in part (iv) is associated with a couple of {\em left} and {\em right minimal indices} equal to $k$ \cite{Gant59}. The set of eigenvalues together with the number, sign characteristics, and sizes of the blocks associated to them in the HKCF($\cH$) in Theorem \ref{HCKF_th}, and the number and sizes of the blocks $\cM_k(\la)$ associated to the minimal indices,
constitute the {\em complete eigenstructure} of $\cH$.

Note that $\cH$ is regular if and only if HCKF($\cH$) does not contain blocks $\cM_k$.

The {\em Hermitian orbit} of the $n\times n$ Hermitian pencil ${\cal H}$, denoted by $\orb^H(\cal H)$, is the set of matrix pencils which are $*$-congruent with ${\cal H}$, namely
$$
\orb^H(\cH):=\{Q^*\cH(\la)Q:\ \ Q\in\CC^{n\times n}\ \mbox{is invertible}\}.
$$
Note that all pencils in $\orb^H(\cH)$ are Hermitian.

The {\em Hermitian bundle} of $\cH$, denoted by $\bun^H(\cH)$, is the set of all Hermitian pencils having the same HKCF as $\cH$ except maybe for the specific values of their distinct finite eigenvalues. Thus, all the pencils in $\bun^H(\cH)$ have the same number of distinct finite eigenvalues and, moreover, there exists an ordering of such distinct finite eigenvalues for which each eigenvalue has the same number and sizes of associated Hermitian canonical blocks (with the same signs associated with the blocks of real eigenvalues).

\vspace{0.2cm}
\begin{remark}
In our definition of Hermitian bundle we allow the finite eigenvalues to vary from one pencil to another in the same bundle. However,  the blocks (with their signs) of the infinite eigenvalue are equal for all pencils in the bundle, in contrast with the standard approach for nonstructured pencils \cite{EdEK97,EdEK99}. The reason for introducing this restriction on the infinite eigenvalue is related to the sign characteristic and to the fact that we expect the Hermitian bundles to have the following property:  if $\cH_1\in\overline{\bun^H}(\cH_2)$ then $\bun^H (\cH_1)\subseteq\overline{\bun^H}(\cH_2)$. This property is necessary for considering the set $\pen_{n\times n}^H$ a stratified manifold whose strata are the bundles, since the closure of a strata must be the finite union of itself with strata of smaller dimensions. This, however, does not hold if we allow finite eigenvalues to become the infinite one in a bundle or vice versa. Let us illustrate this situation in the simple case of $\pen_{1\times 1}^H = \{a + \la b \, : \, a,b \in \mathbb{R}\}$. The possible canonical forms of these $1\times 1$ Hermitian pencils are $+{\cal J}_1(\alpha) = \lambda - \alpha, -{\cal J}_1(\alpha) = -(\lambda - \alpha)$, with $\alpha \in \mathbb{R}$ and finite, $+{\cal J}_1(\infty) = 1, -{\cal J}_1(\infty) = -1$, and $\mathcal{M}_0 = 0$ (the only singular $1\times 1$ pencil). If we include $\la-\alpha = +{\cal J}_1(\alpha)$ for all finite $\alpha \in \mathbb{R}$ and $1 = + {\cal J}_1 (\infty)$ in the same bundle, as might seem natural taking into account the definition of bundles for unstructured pencils, then $+{\cal J}_1 (\infty) = 1\in\overline{\bun^H}(-{\cal J}_1(\beta))$, where $\beta \in \mathbb{R}$ is finite, since $-(\frac{\la}{m}-1)$ converges to $1$ as the natural number $m$ tends to infinity, and $-(\frac{\la}{m}-1)\in\bun^H(-{\cal J}_1(\beta))$. However, $+{\cal J}_1(\alpha) =\la-\alpha\not\in\overline{\bun^H}(-{\cal J}_1(\beta))$ for any $\alpha$.
This means that the previous desired property of bundles does not hold, since $\overline{\bun^H}(-{\cal J}_1(\beta))$ would not include the whole bundle to which $\la-\alpha = +{\cal J}_1(\alpha)$ and $1 = +{\cal J}_1 (\infty)$ belong. Note that the problem remains if $1 = +{\cal J}_1 (\infty)$  is included in the same bundle as $-(\la-\beta) = -{\cal J}_1(\beta)$ for all finite $\beta \in \mathbb{R}$, since taking the sequence $\{\frac{\la}{m} + 1 \} \subset \bun^H ( +{\cal J}_1(\alpha) )$, that tends to $1$ as well, we see that $1 = +{\cal J}_1 (\infty)  \in\overline{\bun^H}(+{\cal J}_1(\alpha))$, but $-(\la-\beta) \not\in\overline{\bun^H}(+{\cal J}_1(\alpha))$ for any $\beta$.  For these reasons, we only allow the finite eigenvalues to vary in the pencils of a given bundle. Using the ideas above, it is possible to construct higher dimensional examples with similar difficulties.
\end{remark}
\vspace{0.2cm}

Next, we introduce a notation that allows us to express some arguments concisely and we state without proof a few very simple properties of Hermitian bundles that are often used. If $\cH \in \pen_{n\times n}^H$ and $\cH_1, \cH_2 \in \bun^H (\cH)$, then we will write
$$
\Hb (\cH_1) \simeq \Hb (\cH_2)
$$
to mean that the $\Hb$s of $\cH_1$ and $\cH_2$ are the same up to the values of their distinct finite eigenvalues.

\begin{lemma} \label{lemm.trivialbundle} Let $\cH_1, \cH_2 \in \pen_{n\times n}^H$ and $Q \in \mathbb{C}^{n\times n}$ be nonsingular. Then
\begin{itemize}
    \item[\rm (a)] $\cH_1 \in \bun^H (\cH_2)$ if and only if $\bun^H (\cH_1) = \bun^H (\cH_2)$,
    \item[\rm (b)] $\cH_1 \in \bun^H (\cH_2)$ if and only if $\Hb (\cH_1) \simeq \Hb (\cH_2)$,
    \item[\rm (c)] $\cH_1 \in \bun^H (\cH_2)$ if and only if $Q^* \cH_1 Q \in \bun^H (\cH_2)$,
    \item[\rm (d)] $\cH_1 \in \overline{\bun^H} (\cH_2)$ if and only if $Q^* \cH_1 Q \in \overline{\bun^H} (\cH_2)$.
\end{itemize}
\end{lemma}

\section{Some technical results}\label{technical_sec}
In this section, we present some results that are needed to prove the main theorems of the paper (in Sections \ref{regular_sec} and \ref{main_sec}).
We first provide a block anti-triangular decomposition of Hermitian matrix pencils. This result can be proven in the same way as the analogous result for symmetric matrix pencils \cite[Theorem 2]{ddd20} but using the factorization $W^*=U^* R$  (equivalently, $W=R^* U$), where $U^*$ is unitary and $R$ is upper-triangular, see e.g., \cite[p. 89, Theorem 2.1.14]{HoJo13}.
\begin{theorem}\label{antitriangular_th}{\rm (Block anti-triangular form of Hermitian pencils).} Let ${\cal H}(\la)$ be a Hermitian pencil. Then, there is a unitary matrix $U$ such that
\begin{equation}\label{antitriangular}
{\cal H}(\la)=U^*\left[\begin{array}{ccc}{\cal A}(\la)&{\cal B}(\la)&{\cal H}_{\mbox{\rm\tiny right}}(\la)\\
{\cal B}(\la)^*&{\cal H}_{\mbox{\rm\tiny reg}}(\la)&0\\{\cal H}_{\mbox{\rm\tiny right}}(\la)^*&0&0\end{array}\right]U,
\end{equation}
where:
\begin{itemize}
\item[\rm(i)] ${\cal A}(\la)$ is a Hermitian pencil.

\item[\rm(ii)] ${\cal H}_{\mbox{\rm\tiny reg}}(\la)$ is a regular Hermitian pencil whose elementary divisors are exactly those of ${\cal H}(\la)$.

\item[\rm(iii)] ${\cal H}_{\mbox{\rm\tiny right}}(\la)$ is a pencil whose complete eigenstructure consists only of the right minimal indices of ${\cal H}(\la)$.
\end{itemize}
As a consequence, ${\cal H}_{\mbox{\rm\tiny right}}(\la)^*$ is a pencil whose complete eigenstructure consists only of the left minimal indices of ${\cal H}(\la)$.
\end{theorem}

In the proofs of the main Theorems \ref{regular_th} and \ref{main_th} we make use of the following results.

\begin{lemma} \label{lemm-direcsum1} Let $\cA_1, \ldots , \cA_s$ and $\cH_1, \ldots , \cH_s$ be Hermitian pencils of different sizes such that, for $i,j = 1, \ldots , s$,
\begin{itemize}
    \item[\rm (a)] the sizes of $\cA_i$ and $\cH_i$ are equal, $\cA_i \in \bun^H \left(\cH_i \right)$, and
    \item[\rm (b)] $\cH_i$ and $\cH_j$ have no finite eigenvalues in common for $i \ne j$.
\end{itemize}
Then \,
$
\cA_1 \oplus \cdots \oplus \cA_s  \in
\overline{\bun^H} \left( \cH_1 \oplus \cdots \oplus \cH_s \right)
$. If, in addition, $\cA_i$ and $\cA_j$ have no finite eigenvalues in common for $i \ne j$, then
$
\cA_1 \oplus \cdots \oplus \cA_s  \in
\bun^H \left( \cH_1 \oplus \cdots \oplus \cH_s \right)
$.
\end{lemma}

\begin{proof} Case 1. Let us assume first that $\cA_i$ and $\cA_j$ have no finite eigenvalues in common for $i \ne j$. Then
\begin{align*}
    \Hb (\cA_1 \oplus \cdots \oplus \cA_s ) & = \Hb (\cA_1) \oplus \cdots \oplus \Hb (\cA_s ) \\
     & \simeq \Hb (\cH_1) \oplus \cdots \oplus \Hb (\cH_s ) = \Hb (\cH_1 \oplus \cdots \oplus \cH_s ) ,
\end{align*}
and Lemma \ref{lemm.trivialbundle}-(b) implies $
\cA_1 \oplus \cdots \oplus \cA_s  \in
\bun^H \left( \cH_1 \oplus \cdots \oplus \cH_s \right)
$.

Case 2. Next, assume that $\cA_i$ and $\cA_j$ have finite eigenvalues in common for some $i\ne j$. Let $\{ \la_1 ,  \ldots ,\la_t\} := \{ \la \in \CC \, : \, \mbox{$\la$ is a finite eigenvalue of  $\cA_i$ and of $\cA_j$ for $i\ne j$}\}$,  where if $\la$ is an eigenvalue of exactly $\ell$ pencils $\cA_i$ (with $\ell>1$) then it is repeated $\ell$ times in the previous set. Since there are infinitely many different sequences of real numbers $\{a_n\}_{n\in\mathbb N}$ with all their terms distinct whose limit is zero (as $m$ tends to infinity), we can choose $t$ of these sequences $\{c_{1}^{(m)}\}, \ldots , \{c_{t}^{(m)}\}$ and replace in $\Hb (\cA_1), \ldots ,\Hb (\cA_s)$ the common finite eigenvalues $\la_1 , \la_2 , \ldots ,\la_t$ by $\la_1 + c_{1}^{(m)},  \ldots , \la_t + c_{t}^{(m)}$
to get sequences $\{\Hb (\cA_{1}^{(m)}) \}, \ldots , \allowbreak \{\Hb (\cA_{s}^{(m)}) \}$ of pencils in $\Hb$s such that, for all $m\in \NN$,
\begin{itemize}
    \item[\rm (p1)] $\Hb (\cA_{i}^{(m)}) \simeq \Hb (\cA_{i}) \in \bun^H \left( \cH_{i} \right)$ for $i= 1\ldots , s$,
    \item[\rm (p2)]  $\Hb (\cA_{i}^{(m)})$ and $\Hb (\cA_{j}^{(m)})$ have no finite eigenvalues in common for $i\ne j$, and
    \item[\rm (p3)] $\lim_{m\rightarrow \infty}\Hb (\cA_{i}^{(m)}) = \Hb (\cA_{i})$, for $i=1, \ldots, s$.
\end{itemize}
Then, the result in Case 1 implies
$$ \Hb (\cA_{1}^{(m)} \oplus \cdots \oplus \cA_{s}^{(m)} ) =
\Hb (\cA_{1}^{(m)}) \oplus \cdots \oplus \Hb (\cA_{s}^{(m)})  \in
\bun^H \left( \cH_1 \oplus \cdots \oplus \cH_s \right)
$$
for all $m$. Then, from (p3) above, $\Hb (\cA_{1} \oplus \cdots \oplus \cA_{s}) \in
\overline{\bun^H} \left( \cH_1 \oplus \cdots \oplus \cH_s \right)$, and the result follows from Lemma \ref{lemm.trivialbundle}-(d).
\end{proof}

\vspace{0.2cm}
\begin{remark} We emphasize that, in Lemma {\rm\ref{lemm-direcsum1}}, the closure in the condition $\cA_1 \oplus \cdots \oplus \cA_s  \in \overline{\bun^H} \left( \cH_1 \oplus \cdots \oplus \cH_s \right)$ cannot be removed. Consider the following $1 \times 1$ Hermitian pencils: $\cA_1 = \cA_2 = \la - 1$, $\cH_1 = \la + 1$, and
$\cH_2 = \la$. Thus, $\cA_1 \in \bun^H \left( \cH_1 \right)$, $\cA_2 \in \bun^H \left( \cH_2 \right)$, and $\cA_1$ and  $\cA_2$ have finite eigenvalues in common. Obviously $\cA_1 \oplus \cA_2 = (\la - 1) \oplus (\la -1) \notin \bun^H \left(\cH_1 \oplus \cH_2\right)$. However, $\cA_1 \oplus \cA_2 = (\la - 1) \oplus (\la -1) \in \overline{\bun^H} \left(\cH_1 \oplus \cH_2 \right)$, since, for each $m =1 , 2, \ldots$,
$(\la - (1-1/m)) \oplus (\la -(1+ 1/m)) \in \bun^H \left(\cH_1 \oplus \cH_2 \right)$. Observe also that in this example, $\bun^H \left(\cH_1 \right) = \bun^H \left(\cH_2 \right)$, but the hypotheses of Lemma {\rm\ref{lemm-direcsum1}} require to take different pencils representing this bundle.
\end{remark}
\vspace{0.2cm}

Lemma \ref{lemm-direcsum1} allows us to prove the following lemma, which is the one we will actually use to prove our main results.

\begin{lemma} \label{lemm-direcsum2} Let $\cH_1, \ldots , \cH_s$ and $\widetilde\cH_1, \ldots , \widetilde\cH_s$ be Hermitian pencils of different sizes such that, for $i,j = 1, \ldots , s$,
\begin{itemize}
\item[\rm (a)] the sizes of $\cH_i$ and $\widetilde \cH_i$ are equal and $\overline{\bun^H}\left( \cH_i \right) \subseteq
\overline{\bun^H}\left( \widetilde \cH_i \right)$, and
\item[\rm (b)] $\widetilde\cH_i$ and $\widetilde\cH_j$ have no finite eigenvalues in common for $i \ne j$.
\end{itemize}
Then \,
$\overline{\bun^H}\left( \cH_1 \oplus \cdots \oplus \cH_s \right) \subseteq
\overline{\bun^H}\left( \widetilde \cH_1 \oplus \cdots \oplus \widetilde \cH_s \right).$
\end{lemma}

\begin{proof} By definition of closure, we only need to prove
$\bun^H \left( \cH_1 \oplus \cdots \oplus \cH_s \right) \subseteq
\overline{\bun^H}\left( \widetilde \cH_1 \oplus \cdots \oplus \widetilde \cH_s \right).$ Let $\cD \in \bun^H \left( \cH_1 \oplus \cdots \oplus \cH_s \right) $. This implies
$
\Hb (\cD)  \simeq \Hb (\cH_1 \oplus \cdots \oplus \cH_s ) =
\Hb (\cH_1) \oplus \cdots \oplus \Hb (\cH_s ).
$
Therefore, $\Hb (\cD) = \Hb (\cD_1) \oplus \cdots \oplus \Hb (\cD_s )$, with $\Hb (\cD_i) \simeq \Hb (\cH_i)$, which, according to Lemma \ref{lemm.trivialbundle} (b), implies $\Hb (\cD_i) \in \bun^H \left( \cH_i \right) \subseteq \overline{\bun^H}\left( \widetilde \cH_i \right)$, for $i =1 , \ldots s$. Thus, there exist sequences of Hermitian pencils, $\{ \cD_i^{(m)}\}  \subseteq \bun^H\left( \widetilde \cH_i \right)$, such that $\lim_{m\rightarrow \infty} \cD_i^{(m)} = \Hb (\cD_i)$,  for $i=1, \ldots s$. In addition, from Lemma \ref{lemm-direcsum1},
$$
\cD_1^{(m)} \oplus \cdots \oplus \cD_s^{(m)} \in \overline{\bun^H}\left( \widetilde \cH_1 \oplus \cdots \oplus \widetilde \cH_s \right)
$$
for all $m$.
Since, $\lim_{m\rightarrow \infty} ( \cD_1^{(m)} \oplus \cdots \oplus \cD_s^{(m)} ) =  \Hb (\cD_1) \oplus \cdots \oplus \Hb (\cD_s ) = \Hb (\cD)$, we get $\Hb (\cD) \in \overline{\bun^H}\left( \widetilde \cH_1 \oplus \cdots \oplus \widetilde \cH_s \right)$, which combined with Lemma \ref{lemm.trivialbundle}-(d) implies $\cD \in \overline{\bun^H}\left( \widetilde \cH_1 \oplus \cdots \oplus \widetilde \cH_s \right)$.
\end{proof}

\vspace{0.2cm}
Lemma \ref{lemm-direcsum2} will be combined  in the proofs of the main results with the next proposition dealing with the canonical blocks appearing in Theorem \ref{HCKF_th}.

\begin{proposition}\label{propo}  Let $\sigma\in\{+1,-1\}$ and $k >0$, $d,d_1,\hdots,d_t\geq0$ be integer numbers.
\begin{itemize}

\item[{\rm(a1)}] If $k$ is even and $a\in\RR$, then
$$
\overline{\bun^H}\left(\sigma{\cal J}_k^H(a) \right) \subseteq  \,\overline{\bun^H}\left(\diag(\cJ_1^H(\mu_1,\overline{\mu_1}),\hdots,\cJ_1^H(\mu_{\frac{k}{2}-1},\overline{\mu_{\frac{k}{2}-1}}),\cJ_1(a_1),-\cJ_1(a_2))\right),
$$
and
$$
\overline{\bun^H}\left(\sigma{\cal J}_k^H(\infty) \right) \subseteq  \, \overline{\bun^H}\left(\diag(\cJ_1^H(\mu_1,\overline{\mu_1}),\hdots,\cJ_1^H(\mu_{\frac{k}{2}-1},\overline{\mu_{\frac{k}{2}-1}}),\sigma\cJ_1(a_1),\sigma\cJ_1(a_2))\right),
$$
with $a_1,a_2\in\RR$, $\mu_1,\hdots,\mu_{\frac{k}{2}-1} \in \CC$ having positive imaginary part, and $a_1,a_2,\mu_1,\allowbreak \hdots,\mu_{\frac{k}{2}-1}$ different to each other.

\item[{\rm(a2)}] If $k$ is odd and $a\in\RR\cup\{\infty\}$, then
$$
\overline{\bun^H}\left( \sigma{\cal J}_k^H(a) \right) \subseteq  \, \overline{\bun^H}\left(\diag(\cJ_1^H(\mu_1,\overline{\mu_1}),\hdots,\cJ_1^H(\mu_{\frac{k-1}{2}},\overline{\mu_{\frac{k-1}{2}}}),\sigma\cJ_1(\widetilde a))\right),
$$
with $\widetilde a\in\RR$, $\mu_1,\hdots,\mu_{\frac{k-1}{2}} \in \CC$ having positive imaginary part, and $\mu_1,\hdots,\mu_{\frac{k-1}{2}}$ different to each other.

\item[{\rm(b)}] If $\mu\in\CC$ has positive imaginary part, then $$\overline{\bun^H}\left(\cJ_{k}^H(\mu,\overline{\mu}) \right) \subseteq  \, \overline{\bun^H}\left(\diag(\cJ_1^H(\mu_1,\overline{\mu_1}),\hdots,\cJ_1^H(\mu_k,\overline{\mu_k}))\right),$$ with $\mu_1,\hdots,\mu_k \in \CC$ being different to each other and having positive imaginary part.

\item[{\rm(c)}] If $\mu\in\CC$ has positive imaginary part, then $\overline{\bun^H} \left( \cJ_k^H(\mu,\overline{\mu})\oplus{\cal M}_d \right) \subseteq\overline{\orb^H}\left({\cal M}_{d+k}\right)$.

\item[{\rm(d)}] $\overline{\orb^H} \left( \diag({\cal M}_{d_1},\hdots,{\cal M}_{d_t}) \right) \subseteq\overline{\orb^H}(\diag(\underbrace{{\cal M}_{\alpha+1},\hdots,{\cal M}_{\alpha+1}}_{\mbox{\tiny $s$}},\underbrace{{\cal M}_{\alpha},\hdots,{\cal M}_{\alpha})}_{\mbox{\tiny $t-s$}}),$ with \break$\sum_{i=1}^{t} d_i=t\alpha+s$ being the Euclidean division of $\sum_{i=1}^{t}d_i$ by $t$.

\end{itemize}
\end{proposition}
\begin{proof} Part (a1). Assume that $k$ is even and $a\in\RR$, and note that $\sigma {\cal J}^H_k(a)$ is the limit of the following sequence of Hermitian pencils (as $m\in\NN$ tends to infinity):
$$
\sigma{\cal S}_{(a,k)}^{(\varepsilon,m)}(\la):=\sigma\left[\begin{array}{cccccc}&&&1&\la-a\\&&1&\la-a\\&\iddots&\iddots&&\\1&\la-a&\\\la-a&&&&\varepsilon/m\end{array}\right]_{k\times k},
$$
with $\varepsilon>0$. By a direct calculation, we get $\det \sigma \cS_{(a,k)}^{(\varepsilon,m)}(\la)=(-1)^{\frac{k}{2}}\left((\la-a)^k-\frac{\varepsilon}{m}\right)$, so the eigenvalues of $\sigma \cS_{(a,k)}^{(\varepsilon,m)}(\la)$ are $a+(\frac{\varepsilon}{m})^{1/k}$, where $(\frac{\varepsilon}{m})^{1/k}$ denote the $k$ different $k$th roots of $\frac{\varepsilon}{m}$. Since $k$ is even and $\frac{\varepsilon}{m}>0$, these consist of $2$ real values and $\frac{k}{2}-1$ pairs of complex conjugate non-real distinct numbers. Therefore,
$$
{\rm HKCF}\left(\sigma\cS_{(a,k)}^{(\varepsilon,m)}\right)= \diag\left(\cJ_1^H(\mu_1,\overline{\mu_1}),\hdots,\cJ_1^H(\mu_{\frac{k}{2}-1},\overline{\mu_{\frac{k}{2}-1}}),\pm\cJ_1(a_1),\pm\cJ_1(a_2)\right),
$$
where $\mu_1,\hdots,\mu_{\frac{k}{2}-1},\overline{\mu_1},\hdots,\overline{\mu_{\frac{k}{2}-1}}$ are the non-real eigenvalues of $\sigma \cS_{(a,k)}^{(\varepsilon,m)}(\la)$, $a_1,a_2$ are the real eigenvalues, and we take $\mu_1,\hdots,\mu_{\frac{k}{2}-1}$ with positive imaginary part. In order to fix the signs in $\left(\pm\cJ_1(a_1),\pm\cJ_1(a_2)\right)$, we will use that the matrix leading terms of $\sigma \cS_{(a,k)}^{(\varepsilon,m)}(\la)$ and of the pencil in the right-hand side of the previous equality are $*$-congruent and so have the same signature. Since the leading term of $\sigma\cS_{(a,k)}^{(\varepsilon,m)}$ is a symmetric $k\times k$ matrix with trace $0$ and eigenvalues $\pm1$, its signature is $(k/2,k/2,0)$. Now, since the signature of the leading term of each block $\cJ_1^H(\mu_i,\overline{\mu_i})$ is $(1,1,0)$, for $i=1,\hdots,\frac{k}{2}-1$, then the signature of the leading coefficient of the remaining block $\diag(\pm\cJ_1(a_1),\pm\cJ_1(a_2))$ must be $(1,1,0)$ as well. As a consequence,
$$
\sigma\cS_{(a,k)}^{(\varepsilon,m)}(\la)\in\bun^H\left(\diag(\cJ_1^H(\mu_1,\overline{\mu_1}),\hdots,\cJ_1^H(\mu_{\frac{k}{2}-1},\overline{\mu_{\frac{k}{2}-1}}),\cJ_1(a_1),-\cJ_1(a_2))\right),
$$
with $a_1,a_2\in\RR$, and $a_1,a_2,\mu_1,\hdots,\mu_{\frac{k}{2}-1}$ being different to each other, as wanted. Since the arguments above are independent of the specific value of $a \in \RR$, using Lemma \ref{lemm.trivialbundle}-(d) we conclude that
$$
\bun^H\left(\sigma{\cal J}_k^H(a) \right) \subseteq
\overline{\bun^H}\left(\diag(\cJ_1^H(\mu_1,\overline{\mu_1}),\hdots,\cJ_1^H(\mu_{\frac{k}{2}-1},\overline{\mu_{\frac{k}{2}-1}}),\cJ_1(a_1),-\cJ_1(a_2))\right),
$$
and the result follows by the definition of closure.

Now, let us prove the claim for $\sigma\cJ_k^H(\infty)$. In this case, we consider the Hermitian perturbation
\begin{equation}\label{infpert}
\sigma {\cal S}_{(\infty,k)}^{(\varepsilon,m)}(\la):=\sigma\begin{bmatrix}
&&\la&1\\
&\iddots&\iddots\\
\la&1&&\\
1&&&\frac{\varepsilon}{m}\la
\end{bmatrix},
\end{equation}
with $\varepsilon>0$. Note that $\sigma {\cal S}_{(\infty,k)}^{(\varepsilon,m)}(\la)$ tends to $\sigma\cJ_k^H(\infty)$ as $m$ tends to infinity. A direct calculation gives $\det(\sigma {\cal S}_{(\infty,k)}^{(\varepsilon,m)}(\la))=(-1)^{k/2}(1-\frac{\varepsilon}{m}\la^k)$, so the eigenvalues of $\sigma {\cal S}_{(\infty,k)}^{(\varepsilon,m)}(\la)$ are the $k$th roots of $\frac{m}{\varepsilon}$. Since $\varepsilon>0$, these are $\frac{k}{2}-1$ pairs of complex conjugate (non-real) numbers and two real numbers. Hence,
$$
{\rm HKCF}\left(\sigma {\cal S}_{(\infty,k)}^{(\varepsilon,m)}\right)=\diag\left(\cJ_1^H(\mu_1,\overline{\mu_1}),\hdots,\cJ_1^H(\mu_{\frac{k}{2}-1},\overline{\mu_{\frac{k}{2}-1}}),\pm\cJ_1(a_1),\pm\cJ_1(a_2)\right),
$$
with $a_1,a_2\in\RR$ and $\mu_1,\hdots,\mu_{\frac{k}{2}-1}$ having positive imaginary part, and all them being different to each other. The signs of $(\pm\cJ_1(a_1),\pm\cJ_1(a_2))$ are determined again by the signatures of the leading terms. The leading term of $\sigma {\cal S}_{(\infty,k)}^{(\varepsilon,m)}$ is
$$
T^{(\varepsilon,m)}:=\sigma\begin{bmatrix}
&&1&0\\
&\iddots&\iddots&\\
1&0&\\
0&&&\frac{\varepsilon}{m}
\end{bmatrix}.
$$
The eigenvalues of $\sigma T^{(\varepsilon,m)}$ are the following:  $(k-2)/2$ of them are equal to $+1$, $(k-2)/2$ are equal to $-1$, another one is equal to $\sigma$, and the last one is equal to $\sigma\frac{\varepsilon}{m}$. Since the signature of the leading term of $\cJ_1^H(\mu,\overline\mu)$ is $(1,1,0)$, the signature of the leading term of $\diag(\pm\cJ_1(a_1),\pm\cJ_1(a_2))$ must be the one of the leading term of  $\diag(\sigma\cJ_1(a_1), \allowbreak\mbox{sign}(\sigma\frac{\varepsilon}{m})\, \cJ_1(a_2))$. Since $\varepsilon>0$  and all the pencils in $\bun^H\left(\sigma{\cal J}_k^H(\infty) \right)$ are congruent to each other, the result follows.

Part (a2). The argument when $a\in\RR$ is similar to the one in part (a1). Thus, we only sketch the main ideas. In this case, $\det(\sigma\cS_{(a,k)}^{(\varepsilon,m)}(\la))=\sigma^k(-1)^{\frac{k-1}{2}}\left((\la-a)^k+\frac{\varepsilon}{m}\right)$, so $\sigma \cS_{(a,k)}^{(\varepsilon,m)}$ has only one real eigenvalue, and the remaining $k-1$ eigenvalues are couples of complex conjugate non-real distinct numbers, namely $\sigma\cS_{(a,k)}^{(\varepsilon,m)}$ is $*$-congruent to $\diag(\cJ_1^H(\mu_1,\overline{\mu_1}),\hdots,\cJ_1^H(\mu_{\frac{k-1}{2}},\overline{\mu_{\frac{k-1}{2}}}),\pm\cJ_1(\widetilde a))$, for some $\mu_1,\hdots,\mu_{\frac{k-1}{2}}$ being different to each other and having positive imaginary part, and $\widetilde a\in\RR$.  The signature of the leading coefficient of $\sigma\cS_{(a,k)}^{(\varepsilon,m)}$ is  $((k+1)/2,(k-1)/2,0)$ if $\sigma=1$ and $((k-1)/2,(k+1)/2,0)$ if $\sigma=-1$. Since, again, the signature of the leading coefficient of $\cJ_1^H(\mu,\overline{\mu})$ is $(1,1,0)$, we conclude that the sign of the block $\cJ_1(a)$ must be equal to $\sigma$.  The fact that the arguments above are again independent of the specific value of $a \in \RR$ allows us to get the desired inclusion of the bundle closures.

For $\sigma \cJ_k^H(\infty)$ in (a2) we can use a similar argument as for the case $k$ even. In this case the eigenvalues of \eqref{infpert} are the $k$th roots of $-\frac{m}{\varepsilon}$, so there is just one real root and $(k-1)/2$ pairs of complex conjugate (non-real) roots. Therefore,
$$
{\rm HKCF}\left(\sigma {\cal S}_{(\infty,k)}^{(\varepsilon,m)}\right)=\diag\left(\cJ_1^H(\mu_1,\overline{\mu_1}),\hdots,\cJ_1^H(\mu_{\frac{k-1}{2}},\overline{\mu_{\frac{k-1}{2}}}),\pm\cJ_1(\widetilde{a})\right),
$$
with $\widetilde{a}\in\RR$ and $\mu_1,\hdots,\mu_{\frac{k-1}{2}}$ being different to each other and having positive imaginary part.  Reasoning as for the $k$ even case with the signatures of the leading terms of $\sigma {\cal S}_{(\infty,k)}^{(\varepsilon,m)}$ and of ${\rm HKCF}\left(\sigma {\cal S}_{(\infty,k)}^{(\varepsilon,m)}\right)$, which are equal to each other, we conclude that the sign of the block $\pm\cJ_1(\widetilde{a})$ must be equal to $\sigma$.

Part (b). Note that, given an arbitrary parameter $\varepsilon>0$, the Hermitian pencil
$$
\begin{array}{ccl}
\cJ_k^H(\mu,\overline{\mu})+T^{(k,m,\varepsilon)} &:=&\left[\begin{array}{cccc|cccc}
     &&&&&&1&\la-\overline\mu\\&&&&&\iddots&\iddots&\\&&&&1&\la-\overline\mu&&\\&&&&\la-\overline{\mu}&&& \\\hline &&1&\la-\mu&&&&\\&\iddots&\iddots&&&&&\\1&\la-\mu&&&&&&\\\la-\mu&&&&&&&
\end{array}\right]\\&&+\left[\begin{array}{cccc|cccc}&&&&&&&\frac{\varepsilon}{m}\\&&&&&&\iddots&\\&&&&&\frac{(k-1)\varepsilon}{m}&&\\&&&&\frac{k\varepsilon}{m}&&&\\\hline&&&\frac{k\varepsilon}{m}&&&&\\&&\frac{(k-1)\varepsilon}{m}&&&&&\\&\iddots&&&&&&\\\frac{\varepsilon}{m}&&&&&&&\end{array}\right]
\end{array}
$$
has eigenvalues $\mu-\frac{\varepsilon}{m},\overline{\mu-\frac{\varepsilon}{m}},\hdots,\mu-k\frac{\varepsilon}{m},\overline{\mu-k\frac{\varepsilon}{m}}$, which are all different to each other, so
$$
\cJ_k^H(\mu,\overline{\mu})+T^{(k,m,\varepsilon)}\in\bun^H\left(\diag(\cJ_1^H(\mu_1,\overline{\mu_1}),\hdots,\cJ_1^H(\mu_k,\overline{\mu_k})\right),
$$
with $\mu_1,\hdots,\mu_k$ being different to each other and having positive imaginary part. Then $\lim_{m\rightarrow \infty}\cJ_k^H(\mu,\overline{\mu})+T^{(k,m,\varepsilon)} = \cJ_k^H(\mu,\overline{\mu})$, which implies that $\cJ_k^H (\mu,\overline{\mu}) \in \overline{\bun^H}\left(\diag(\cJ_1^H(\mu_1,\overline{\mu_1}),\hdots,\cJ_1^H(\mu_k,\overline{\mu_k})\right)$. Since this is valid for any $\mu\in\CC$ with positive imaginary part, we get that $\bun^H(\cJ_k^H (\mu,\overline{\mu}) ) \subseteq \overline{\bun^H}\left(\diag(\cJ_1^H(\mu_1,\overline{\mu_1}),\allowbreak \hdots, \allowbreak \cJ_1^H(\mu_k,\overline{\mu_k})\right),$ using Lemma \ref{lemm.trivialbundle}-(d), which implies the result because of the definition of closure.

Part (c). First note that, if $P$ is the block permutation matrix
$$
P=\begin{bmatrix}
I_k&0&0&0\\
0&0&I_k&0\\
0&I_{d+1}&0&0\\
0&0&0&I_d\\
\end{bmatrix},
$$
then 
\begin{equation*}
\begin{aligned}
P^*\Big(&\cJ_k^H(\mu,\overline{\mu})\oplus{\cal M}_d(\la)\Big)P \\
&=\small \begin{bmatrix}
I_k&0&0&0\\
0&0&I_{d+1}&0\\
0&I_k&0&0\\
0&0&0&I_d\\
\end{bmatrix}
\begin{bmatrix}0&\cJ_k^H(\overline\mu)&0&0\\\cJ_k^H(\mu)&0&0&0\\0&0&0&\cL_d(\la)^\top\\0&0&\cL_d(\la)&0\end{bmatrix}
\begin{bmatrix}
I_k&0&0&0\\
0&0&I_k&0\\
0&I_{d+1}&0&0\\
0&0&0&I_d\\
\end{bmatrix}
\\
&=\begin{bmatrix}0&0&\cJ_k^H(\overline\mu)&0\\0&0&0&\cL_d(\la)^\top\\\cJ_k^H(\mu)&0&0&0\\0&\cL_d(\la)&0&0\end{bmatrix}
\\
&=\begin{bmatrix}
0&\cJ_k^H(\overline{\mu})\oplus{\cal L}_d(\la)^\top \\
\cJ_k^H(\mu)\oplus{\cal L}_d(\la)&0\\
\end{bmatrix}.
\end{aligned}
\end{equation*}
Therefore, $\cJ_k^H(\mu,\overline{\mu})\oplus{\cal M}_d(\la)$ is $*$-congruent to $\left[\begin{smallmatrix}
0&\cJ_k^H(\overline{\mu})\oplus{\cal L}_d(\la)^\top \\
\cJ_k^H(\mu)\oplus{\cal L}_d(\la)&0\\
\end{smallmatrix}\right]$.

By \cite[Section 5.1]{Bole98} (or \cite[Lemma 5]{Pokr86}), there exist two invertible matrices $R$ and $Q$ and an arbitrarily small (entry-wise for each coefficient) matrix pencil ${\cal E}$ such that
\begin{equation*}
R\left( \cJ_k^H(\mu)\oplus{\cal L}_d(\la)+{\cal E} \right) Q = {\cal L}_{d+k}(\la).
\end{equation*}
Applying the $*$ operator to both sides of the previous equality, we obtain
\begin{equation*}
Q^*\left( \cJ_k^H(\overline{\mu})\oplus{\cal L}_d(\la)^\top
+{\cal E}^* \right) R^{*} = {\cal L}_{d+k}(\la)^\top.
\end{equation*}
Combining these two identities we obtain
\begin{equation*}
\begin{aligned}
\begin{bmatrix}
Q^*&0\\
0&R\\
\end{bmatrix}
\left(
\begin{bmatrix}
0&\cJ_k^H(\overline{\mu})\oplus{\cal L}_d(\la)^\top\\
\cJ_k^H(\mu)\oplus{\cal L}_d(\la)&0\\
\end{bmatrix}
+
\begin{bmatrix}
0&{\cal E}^*\\
{\cal E}&0\\
\end{bmatrix}
 \right)
\begin{bmatrix}
Q &0\\
0&R^*\\
\end{bmatrix} \\=
\begin{bmatrix}
0&{\cal L}_{d+k}(\la)^\top\\
{\cal L}_{d+k}(\la)&0\\
\end{bmatrix}.
\end{aligned}
\end{equation*}
Since $\cal E$ is arbitrarily small, this implies that $\left[\begin{smallmatrix}
0&\cJ_k^H(\overline{\mu})\oplus{\cal L}_d(\la)^\top \\
\cJ_k^H(\mu)\oplus{\cal L}_d(\la)&0\\
\end{smallmatrix}\right]\in\overline{\orb^H}(\cM_{d+k)}$, so $\cJ_k^H(\mu,\overline{\mu})\oplus\cM_d(\la)\in\overline{\orb^H}(\cM_{d+k})$ as well. Again, the fact that this is valid for any $\mu \in \CC$ with positive imaginary part allows us to get the inclusion of the corresponding bundle-orbit closures.

Part (d). Note that $\diag({\cal M}_{d_1},\hdots,{\cal M}_{d_t}) \in \overline{\orb^H}(\diag({\cal M}_{\alpha+1},\hdots,{\cal M}_{\alpha+1}, \allowbreak {\cal M}_{\alpha},\hdots,{\cal M}_{\alpha}))$ (with the number of blocks in the statement) was already proven in the proof of Theorem 3 in \cite{ddd20}. This result immediately implies the inclusion of the corresponding orbit closures.
\end{proof}

The final technical result of this section is Lemma \ref{lemm.fromlowrank}, which will allow us to simplify the proofs of our main results (Theorems \ref{regular_th} and \ref{main_th}).

\vspace{0.2cm}

\begin{lemma} \label{lemm.fromlowrank} Let $\widetilde r, r$, and $n$ be nonnegative integers such that $0\leq \widetilde r < r \leq n$. Let $\widetilde \cH \in \pen_{n\times n}^H$ be a Hermitian pencil with $\rank \widetilde \cH = \widetilde r$. Then, there exists $\cH \in \pen_{n\times n}^H$ with $\rank \cH =  r$ such that $\overline{\bun^H} ( \widetilde \cH) \subseteq \overline{\bun^H} (\cH)$.
\end{lemma}

\begin{proof} Note that the $\Hb (\widetilde \cH)$ has at least one block $\cM_d$. Thus, in the first part of the proof, we consider the sequence of perturbed blocks $\cM_d^{(m)} := \cM_d + \frac{1}{m}  E_d$, where $m = 1,2,\ldots$, and $E_d$ is a $(2d +1)\times (2d +1)$ constant matrix that has $1$ in the $(1,1)$-entry and zeroes elsewhere. It is immediate to prove that $\det \cM_d^{(m)} = \pm \frac{1}{m}$, which implies that $\cM_d^{(m)}$ is regular with all its eigenvalues equal to $\infty$. Moreover the rank of the leading coefficient of $\cM_d^{(m)}$ is $2d$, which implies that $\Hb (\cM_d^{(m)}) = \pm \cJ_{2d+1}^H(\infty)$.
In order to determine the correct sign $\sigma$ note that the signatures of the zero degree coefficients of
$\cM_d^{(m)}$ and $\sigma \cJ_{2d+1}^H(\infty)$ must be the same. The signature of the zero degree coefficient of $\cM_d^{(m)}$ is $(d+1,d,0)$, while the one of the zero degree coefficient of $\cJ_{2d+1}^H(\infty)$ is also $(d+1,d,0)$. Thus, the correct sign is $\sigma = +1$ and we have proved that, for all $m =1, 2,\ldots$,
\begin{equation} \label{eq.mde1}
\Hb(\cM_d^{(m)}) = \Hb( \cM_d + \frac{1}{m}  E_d )= \cJ_{2d+1}^H(\infty).
\end{equation}

Next, note that it is enough to find a Hermitian pencil $\cH$ with $\rank \cH = r$ such that $\bun^H ( \widetilde \cH) \subseteq \overline{\bun^H} (\cH)$ due to the definition of closure. Let $\widehat \cH \in \bun^H ( \widetilde \cH)$. Then
$$
\Hb (\widehat \cH ) = \left(\bigoplus_{i=1}^{n-r} \cM_{d_i} \right) \oplus \left(\bigoplus_{i=n-r+1}^{n-\widetilde r} \cM_{d_i} \right) \oplus \Hb_{reg} (\widehat \cH),
$$
where $\Hb_{reg} (\widehat \cH)$ includes all the blocks of types (i), (ii), and (iii) in Theorem \ref{HCKF_th} of the $\Hb$ of $\widehat \cH$. Observe that, for all the pencils in $\bun^H ( \widetilde \cH)$, the parameters $d_1, \ldots , d_{n-\widetilde r}$ are the same and $\Hb_{reg} (\widehat \cH)\simeq \Hb_{reg} (\widetilde \cH)$. Then, let us define the sequence of pencils
$$
\cH^{(m)} = \left(\bigoplus_{i=1}^{n-r} \cM_{d_i} \right) \oplus \left(\bigoplus_{i=n-r+1}^{n-\widetilde r} \left( \cM_{d_i} + \frac{1}{m} E_{d_i} \right) \right) \oplus \Hb_{reg} (\widehat \cH),
$$
that satisfies $\lim_{m\rightarrow \infty} \cH^{(m)}
= \Hb (\widehat \cH )$. Moreover \eqref{eq.mde1} implies that, for all $m\in\mathbb N$, $\rank \cH^{(m)} = r$ and
\begin{equation} \label{eq.bigoplusrank}
\Hb(\cH^{(m)}) = \left(\bigoplus_{i=1}^{n-r} \cM_{d_i} \right) \oplus \left(\bigoplus_{i=n-r+1}^{n-\widetilde r} \cJ_{2d_i+1}^H(\infty)  \right) \oplus \Hb_{reg} (\widehat \cH),
\end{equation}
which is independent of $m$. From the right-hand side of \eqref{eq.bigoplusrank}, we define
$$
\cH := \left(\bigoplus_{i=1}^{n-r} \cM_{d_i} \right) \oplus \left(\bigoplus_{i=n-r+1}^{n-\widetilde r} \cJ_{2d_i+1}^H(\infty)  \right) \oplus \Hb_{reg} (\widetilde \cH),
$$
which is independent of $m$ and of the particular $\widehat \cH$ in
$\bun^H ( \widetilde \cH)$ we are considering. Observe that $\rank \cH = r$ and that, since  $\Hb_{reg} (\widehat \cH)\simeq \Hb_{reg} (\widetilde \cH)$, then $\cH^{(m)} \in \bun^H ( \cH)$ for all $m$. Since $\lim_{m\rightarrow\infty}\cH^{(m)}=\Hb(\widehat \cH)$, we conclude that $\Hb (\widehat \cH ) \in \overline{\bun^H} ( \cH)$, which combined with Lemma \ref{lemm.trivialbundle}-(d) implies  $\widehat \cH  \in \overline{\bun^H} ( \cH)$, and this proves the result.
\end{proof}

\section{The case of general Hermitian pencils}\label{regular_sec}
In Theorem \ref{regular_th} we present the generic eigenstructures of complex Hermitian $n\times n$ pencils, and we describe the set $\pen_{n\times n}^H$ as a finite union of closed sets, which are the closures of the bundles corresponding to these generic eigenstructures.  Theorem \ref{regular_th} is the first main result of this paper.

\begin{theorem}{\rm (Generic complete eigenstructures of Hermitian matrix pencils).} \label{regular_th}
Let $n \geq 2$, $0\leq d\leq\lfloor\frac{n}{2}\rfloor$, and $0\leq c\leq  n-2d$. Let us define the following complex Hermitian $n\times n$ regular matrix pencils:
\begin{equation}\label{max-real}
\begin{array}{ccl}
{\cal
R}_{c,d} (\lambda)&:=&\diag(\cJ_1^H(\mu_1,\overline{\mu_1}),\hdots,\cJ_1^H(\mu_d,\overline{\mu_d}),\\&& \cJ_1(a_1),\dots , \cJ_1(a_c),-\cJ_1(a_{c+1}),\hdots,-\cJ_1(a_{n-2d}) ),\,
\end{array}
\end{equation}
where $a_1,\hdots ,a_{n-2d} \in\RR$, $\mu_1,\hdots,\mu_d \in \CC \setminus \RR$ have positive imaginary part, $a_i \neq a_j$, and $\mu_i\neq\mu_j$, for $i\neq j$. Then:
\begin{enumerate}
\item[\rm (i)] For every complex Hermitian $n\times n$ matrix pencil ${\cal H}(\lambda)$ there exist integers $c$ and $d$ such that
$\overline{\bun^H} (\cH)\subseteq  \, \overline{\bun^H}({\cal R}_{c,d})$. 

\item[\rm (ii)] $\bun^H({\cal R}_{c',d'})
\cap \overline{\bun^H}({\cal R}_{c,d})=\emptyset$ whenever $d \ne
d'$ or $c\neq c'$.
\item[\rm (iii)] The set $\pen_{n\times n}^H$ is equal to $\displaystyle \bigcup_{\begin{array}{c}\mbox{\tiny $0\leq d \leq \lfloor\frac{n}{2}\rfloor$}\\\mbox{\tiny$0\leq c\leq n-2d$}\end{array}} \overline{\bun^H}({\cal R}_{c,d} )$.
\end{enumerate}
\end{theorem}
\begin{proof}
Claim (iii) is an immediate consequence of claim (i).

Let us prove (i). As a consequence of Lemma \ref{lemm.fromlowrank}, with $r=n$, we only need to prove it when $\cH$ is a regular pencil, i.e., when $\cH$ has rank exactly $r$. The HKCF of any Hermitian $n\times n$ regular pencil $\cH$ is a direct sum of blocks of the form $\cJ_k^H(\mu,\overline{\mu})$ and $\sigma\cJ_\ell^H(a)$, with $\mu$ having positive imaginary part and $a\in\RR\cup\{\infty\}$. By combining claims (a1)--(a2)--(b) in Proposition \ref{propo} with Lemma \ref{lemm-direcsum2}, the closure of the Hermitian bundle corresponding to $\cH$ is included in the closure of the Hermitian bundle of a direct sum of blocks $\cJ_1^H(\mu_i,\overline{\mu_i})$ and $\pm\cJ_1(a_i)$, where $\mu_i$ has positive imaginary part and $a_i \in \RR$, as long as we take all the eigenvalues in this last direct sum to be distinct. This is always allowed by Proposition \ref{propo} since the parameters appearing in the bundles in that result are arbitrary (subjected to their defining conditions).

Now, let us prove (ii). We are going to prove it first when $d\neq d'$. By contradiction, assume that $\cH(\la)$ is any pencil in $\bun^H({\cal R}_{c',d'})$ but $\cH
\in\overline{\bun^H}({\cal R}_{c,d})$, with ${\cal R}_{c',d'}(\la)$ as in the statement. Then, there is a sequence of pencils in $\bun^H({\cal R}_{c,d})$, say $\{\cH_m(\la)\}$, converging to $\cH(\la)$. Since $\cH_m\in\bun^H({\cal R}_{c,d})$, the eigenvalues of $\cH_m$ are of the form $\mu_{1,m},\overline{\mu_{1,m}},\hdots,\mu_{d,m},\overline{\mu_{d,m}}$, and $a_{1,m},\hdots,a_{n-2d,m}$, with $\mu_{1,m},\hdots,\mu_{d,m}$ having positive imaginary part, $a_{1,m},\hdots,a_{n-2d,m}\in\RR$, and $\mu_{i,m}\neq\mu_{j,m}$, $a_{i,m}\neq a_{j,m}$, for $i\neq j$. Analogously, the eigenvalues of $\cH$ are $\mu_1',\overline{\mu_1'},\hdots,\mu_{d'}',\overline{\mu_{d'}'}$, and $a'_1,\hdots,a'_{n-2d'}\in\RR$, with $\mu_1',\hdots,\mu_{d'}'$ having positive imaginary part, and all these numbers being different to each other.

Let us assume first that $d'>d$. By the continuity of the eigenvalues of regular pencils (see, for instance, Theorem 2.1 in \cite[Ch. 6]{stewart-sun}), at least one of the sequences $\{a_{i,m}\}$ converges to some $\mu_j'$. But this is not possible, since $a_{i,m}\in\RR$ and $\mu_j'\in\CC\setminus\RR$.

Now, assume that $d>d'$. By the continuity of the eigenvalues again, there is at least one sequence $\{\mu_{i,m}\}$ converging to some $a_j'$. Since $a_j'\in\RR$, this implies that ${\rm im}(\mu_{i,m})$ converges to $0$, and this in turn implies that $\overline{\mu_{i,m}}$ converges to $a_j'$ as well. Therefore, the limit of $\{\cH_m\}$ (namely, a pencil with HKCF equal to $\cR_{c',d'}$) has $a_j'$ as an eigenvalue with algebraic multiplicity at least $2$. This is in contradiction with the fact that the real eigenvalues of $\cR_{c',d'}$ are different to each other.

It remains to prove the statement when $d=d'$ but $c\neq c'$. Note that the signature of the leading coefficient of $\cJ_1(\mu,\overline\mu)$ (namely, $\left[\begin{smallmatrix}0&1\\1&0\end{smallmatrix}\right]$) is $(1,1,0)$. Therefore, the signature of the leading coefficient of $\cR_{c',d}$ is $(c'+d,n-c'-d,0)$, which is equal to the signature of the leading coefficient of $\cH (\la)$ since the signature is invariant under $*$-congruence. On the other hand,  using again that the signature is invariant under $*$-congruence, the signature of the leading coefficient of $\cH_m$ is the same as the signature of the leading coefficient of $\cR_{c,d}$, namely $(c+d,n-c-d,0)$. Then, $c'\neq c$ implies that either $c+d<c'+d$ or $n-c-d<n-c'-d$, which contradicts Theorem 4.3 in \cite{LaRo05}.
\end{proof}

The following result is an immediate consequence of Theorem \ref{regular_th}.

\begin{corollary}\label{gencount_coro}
The number of different generic eigenstructures in $\pen_{n\times n}^H$ is equal to $\displaystyle\left(\left\lfloor\frac{n}{2}\right\rfloor+1\right)\left\lfloor\frac{n+3}{2}\right\rfloor$.
\end{corollary}
\begin{proof}
The number of different pairs $(c,d)$ in Theorem \ref{regular_th} is
$$
\begin{array}{ccl}
\displaystyle\sum_{d=0}^{\lfloor n/2\rfloor}(n-2d+1)&=&(n+1)\left(\lfloor\frac{n}{2}\rfloor+1\right)-\lfloor\frac{n}{2}\rfloor\left(\lfloor\frac{n}{2}\rfloor+1\right)\\&=&\left(\lfloor\frac{n}{2}\rfloor+1\right)\left(n+1-\lfloor\frac{n}{2}\rfloor\right)=\left(\lfloor\frac{n}{2}\rfloor+1\right)\lfloor\frac{n+3}{2}\rfloor.\end{array}
$$
\end{proof}

Note that the number of generic eigenstructures in $\pen_{n\times n}^H$, according to Corollary \ref{gencount_coro}, is larger than $1$, for all $n\in\NN$.  This is a consequence of the different forms to distribute in regular Hermitian pencils the eigenvalues between real and non-real ones and also of the different possible sign characteristics of the real eigenvalues. The fact that there are more than one (in fact, many for large $n$) generic eigenstructures of regular Hermitian pencils is in stark contrast with the number of generic eigenstructures of regular unstructured pencils and of regular complex symmetric pencils, which have only one generic eigenstructure corresponding to all the eigenvalues being different and simple.

\section{The case of Hermitian pencils with bounded rank}\label{main_sec}

In Theorem \ref{main_th},  which is the second main result of this paper, we prove that $\pen_{n\times n}^H(r)$ is the union of a finite number of closures of Hermitian bundles, and we explicitly provide the HKCF of each of these bundles. As a consequence, these HKCFs are the generic canonical forms of complex Hermitian $n\times n$ matrix pencils with rank at most $r$. In other words, they provide the generic complete eigenstructures of these pencils.  Surprisingly, none of these eigenstructures contain non-real eigenvalues, unlike what happens with the generic regular complete eigenstructures for general Hermitian pencils provided in Theorem \ref{regular_th}.

\begin{theorem}{\rm (Generic complete eigenstructures of Hermitian matrix pencils with bounded rank).} \label{main_th}
Let $n$ and $r$ be integers such that $n \geq 2$ and $1\leq r \leq n-1$. Set $d=0,1,\ldots,\lfloor\frac{r}{2}\rfloor\, $ and let $d=(n-r)\alpha+s$ be the Euclidean division of $d$ by $n-r$. Let us define the following complex Hermitian $n\times n$ matrix pencils with rank $r$:
\begin{equation}\label{max}
\begin{array}{ccl}
{\cal
K}_{c,d} (\lambda)&:=&\diag(\overbrace{{\cal M}_{\alpha+1},\hdots,{\cal M}_{\alpha+1}}^{s},
\overbrace{{\cal M}_{\alpha},\hdots,{\cal M}_{\alpha}}^{n-r-s},\\&& \cJ_1(a_1),\dots , \cJ_1(a_c),-\cJ_1(a_{c+1}),\hdots,-\cJ_1(a_{r-2d}) ),\,
\end{array}
\end{equation}
where $a_1,\hdots ,a_{r-2d} \in\RR, a_i \neq a_j$ for $i\neq j$, and $c=0,1,\hdots,r-2d$. 
Then:
\begin{enumerate}
\item[\rm (i)] For every complex Hermitian $n\times n$ matrix pencil ${\cal H}(\lambda)$ with rank
at most $r$, there exist nonnegative integers $c$ and $d$, with $0\leq d\leq\lfloor\frac{r}{2}\rfloor$ and $0\leq c\leq r-2d$, such that
$\overline{\bun^H}({\cal H})\subseteq\overline{\bun^H}({\cal K}_{c,d})$.


\item[\rm (ii)] $\bun^H ({\cal K}_{c',d'})\cap \overline{\bun^H}({\cal K}_{c,d}) = \emptyset$ whenever $d \ne
d'$ or $c\neq c'$.
\item[\rm (iii)] The set $\pen_{n\times n}^H(r)$ is a closed subset of $\pen^{H}_{n \times n}$, and it is equal to $\displaystyle \bigcup_{\begin{array}{c}\mbox{\tiny $0\leq d \leq \lfloor\frac{r}{2}\rfloor$}\\\mbox{ \tiny$0\leq c\leq r-2d$}\end{array}} \overline{\bun^H}({\cal K}_{c,d} )$.
\end{enumerate}
\end{theorem}

\begin{proof} The proof has the same structure as the one for Theorem 3 in \cite{ddd20}, adapted to Hermitian pencils instead of symmetric ones. However, we emphasize that some interesting differences also appear related to the role played by the sign characteristic and by the blocks $\cJ_{k}^H (\mu,\overline{\mu})$ corresponding to pairs of non-real complex conjugate eigenvalues.

The set $\pen_{n\times n}^H(r)$ is a closed subset of $\pen^{H}_{n \times n}$ because it is the intersection of $\pen^{H}_{n \times n}$ with the set of complex $n\times n$ matrix pencils with rank at most $r$, which is a closed set. Therefore, claim (iii) is an immediate consequence of (i), so we only need to prove (i) and (ii).

Let us start proving (i). Because of Lemma \ref{lemm.fromlowrank}, we can restrict ourselves to Hermitian $n\times n$ pencils with rank exactly $r$.  So let $\cH(\la)$ be a Hermitian pencil with $\rank\cH=r$. By Theorem \ref{HCKF_th}, we may assume that
$$
{\rm HKCF}(\cH)=\diag\left(\bigoplus_t\sigma_t \cJ_{k_t}^H(b_t),\bigoplus_j\cJ_{k_j}^H(\la_j,\overline{\la_j}),\bigoplus_\ell\cM_{d_\ell}\right),
$$
with $b_t\in\RR\cup\{\infty\}$ and  $\la_j\in\CC$ with positive imaginary part (the number of blocks $\cJ_{k_t}^H(b_t)$, $\cJ_{k_j}^H(\la_j,\overline{\la_j})$ is not relevant, and the number of blocks $\cM_{d_\ell}(\la)$ is $n-r$, by the rank-nullity Theorem). Then, by (a1), (a2), and (b) in Proposition \ref{propo}, together with Lemma \ref{lemm-direcsum2} and Lemma \ref{lemm.trivialbundle}, $\overline{\bun^H}(\cH)$ is included in the closure of the Hermitian bundle of
$$
\widehat\cH(\la):=\diag\left(\bigoplus_{i=1}^c\cJ_1(a_i) ,\bigoplus_{i=c+1}^{r-2d}(-\cJ_1(a_i)),\bigoplus_p\cJ_1^H(\mu_p,\overline{\mu_p}),\bigoplus_\ell\cM_{d_\ell}(\la)\right),
$$
for some $a_i\in\RR$ and $\mu_p$ having positive imaginary part, and for some $1\leq c\leq r-2d$, with $0\leq d\leq\lfloor\frac{r}{2}\rfloor$. Moreover, since there are infinitely many possible choices for the distinct eigenvalues in the bundles of the right-hand side of each of the inclusions in parts (a1), (a2), and (b) in Proposition \ref{propo}, the values $a_i$ can be taken to be different to each other, for $i=1,\hdots,r-2d$, and the same happens for the values $\mu_p$. Since $r\leq n-1$, there is at least one block $\cM_{d_1}(\la)$ in the previous direct sum defining $\widehat\cH(\la)$. Then, by repeatedly applying  Proposition \ref{propo}--(c) and Lemma \ref{lemm-direcsum2} (as many times as the number of $\cJ^H_{1}(\mu_p,\overline{\mu_p})$ blocks in $\widehat\cH$), the closure of the Hermitian bundle of $\widehat\cH$ is included in the closure of the Hermitian  bundle of
\begin{equation} \label{eq-auxxfro}
\widetilde \cH(\la):=\diag\left(\bigoplus_{i=1}^c\cJ_1(a_i) ,\bigoplus_{i=c+1}^{r-2d}(-\cJ_1(a_i)),\cM_{d_1+q}(\la),\bigoplus_{\ell>1}\cM_{d_\ell}(\la)\right),
\end{equation}
for some $q\geq0$ (which is equal to the number of all blocks in  $\bigoplus_p\cJ_1^H(\mu_p,\overline{\mu_p})$).  To get this inclusion we first split  (modulo permutation of direct summands) $\widehat\cH$ into $\cH_1$ and $\cH_2$, where $\cH_2$ is the direct sum of $\cM_{d_1}$ with $\cJ^H_{1}(\mu_p,\overline{\mu_p})$ which, for simplicity, is the last block of this kind in $\widehat\cH$, and $\cH_1$ is the direct sum of the remaining blocks. Then, by part (c) in Proposition \ref{propo} and Lemma \ref{lemm-direcsum2}, $\overline{\bun^H}(\widehat \cH)=\overline{\bun^H}(\cH_1 \oplus \cH_2)\subseteq\overline{\bun^H}(\widetilde\cH_1 \oplus \widetilde\cH_2)$, with $\widetilde\cH_1=\cH_1$ and $\widetilde\cH_2=\cM_{d_1+1}$, since $\widetilde\cH_1$ and $\widetilde\cH_2$ do not have common eigenvalues. Now, we proceed in the same way with the new pencil $\widetilde\cH_1 \oplus \widetilde\cH_2$, and so on until we get $\widetilde \cH(\la)$ in \eqref{eq-auxxfro}. Applying claim (d) in Proposition \ref{propo},  together with Lemma \ref{lemm-direcsum2} again,  to the $\cM_d$ blocks of $\widetilde \cH(\la)$, we conclude that the closure of the Hermitian bundle of $\widetilde\cH$ is in turn included in the closure of the Hermitian bundle of
$$
\diag\left(\bigoplus_{i=1}^c\cJ_1(a_i) ,\bigoplus_{i=c+1}^{r-2d}(-\cJ_1(a_i)),\bigoplus\cM_{\alpha+1}(\la),\bigoplus\cM_{\alpha}\right),
$$
for some fixed $\alpha$, and where the total number of blocks $\cM_\alpha$ and $\cM_{\alpha+1}$ in the previous direct sum is $n-r$. Taking this into account, if $s$ is the number of blocks $\cM_{\alpha+1}$, then the number of blocks $\cM_\alpha$ is $n-r-s$. The value of $\alpha$ can be obtained by adding up the number of rows (or columns) in the previous pencil and equating to $n$, namely:
$$
n=r-2d+s(2\alpha+3)+(n-r-s)(2\alpha+1)=n+2(\alpha(n-r)+s-d),
$$
which implies $d=\alpha(n-r)+s$, as claimed.

Summarizing, we have proved that $\overline{\bun^H}(\cH)\subseteq\overline{\bun^H}(\cK_{c,d})$, for some $c,d$ as in the statement. This proves (i).

Now, let us prove (ii). First, we need to see that if $d'>d$, then for any $\cH \in \bun^H({\cal K}_{c',d'})$,  $\cH \notin \overline{\bun^H}({\cal K}_{c,d})$, for any $c$ and $c'$. By Lemma \ref{lemm.trivialbundle}-(d) and the definition of bundle, we can take $\cH = {\cal K}_{c',d'}$ for certain distinct real numbers $a_1, \ldots , a_{r - 2 d'}$. Then, the same argument as the one in the proof of Theorem 3 in \cite{ddd20} is still valid, i.e., ${\cal K}_{c',d'} \in \overline{\bun^H}({\cal K}_{c,d})$ would be against the  majorization of the Weyr sequence of right minimal indices, see, for instance, \cite[Lemma 1.2]{Hoyo90}.

It remains to prove that if $d'<d$ or $d'=d$ but $c\neq c'$, then $\bun^H({\cal K}_{c',d'}) \cap \overline{\bun^H}({\cal K}_{c,d}) = \emptyset$ too.

The first part of the rest of the proof is almost the same as the proof of Theorem 3 in \cite{ddd20}, replacing $\top$ in that proof by $^*$. However, there are several relevant differences, so we include here the proof for completeness and for the sake of reader.

By contradiction, if $\bun^H({\cal K}_{c',d'}) \cap \overline{\bun^H}({\cal K}_{c,d}) \neq \emptyset$, then at least one pencil $S(\la)$ congruent to ${\cal K}_{c',d'}(\la)$ as in \eqref{max}, with $a_i\neq a_j$, for $i\neq j$, and $a_i\in\RR$, must be the limit of a sequence of pencils in $\bun^H({\cal K}_{c,d})$. Let $\{{\cal S}_m(\la)\}_{m\in\NN}$ be a sequence of pencils with ${\cal S}_m(\la)\in\bun^H({\cal K}_{c,d})$, for all $m\in\NN$. Then, by Theorem \ref{antitriangular_th}

\begin{equation}\label{Sm}
{\cal S}_m(\la)=Q_m^*\left[\begin{array}{ccc}{\cal A}_m(\la)&{\cal B}_m(\la)&{\cal S}_{\tiny\mbox{\rm right}}^{(m)}(\la)\\{\cal B}_m(\la)^*&S_{\tiny\mbox{\rm reg}}^{(m)}(\la)&0\\{\cal S}_{\tiny\mbox{\rm right}}^{(m)}(\la)^*&0&0\end{array}\right]Q_m,
\end{equation}
with $Q_m\in\CC^{n\times n}$ being a unitary matrix, for all $m\in\NN$, and
\begin{itemize}
\item ${\cal S}_{\tiny\mbox{\rm right}}^{(m)}(\la)$ has size $d\times (n-r+d)$, and complete eigenstructure consisting of the right minimal indices of ${\cal K}_{c,d}(\la)$,
\item ${\cal S}_{\tiny\mbox{\rm right}}^{(m)}(\la)^*$ has size $(n-r+d)\times d$, and complete eigenstructure consisting of the left minimal indices of ${\cal K}_{c,d}(\la)$,
\item ${\cal S}^{(m)}_{\tiny\mbox{\rm reg}}(\la)$ is a regular (Hermitian) pencil of size $(r-2d)\times(r-2d)$, with $r-2d$ distinct real eigenvalues and $c$ of them having positive sign characteristic.
\end{itemize}

Now, let us assume that ${\cal S}_m(\la)$ converges to some pencil ${\cal S}(\la) \in \bun^H({\cal K}_{c',d'})$. Since the set of unitary $n\times n$ matrices is a compact subset of the metric space $(\CC^{n\times n},\|\cdot\|_2)$, the sequence $\{Q_m\}_{m\in\NN}$ has a convergent subsequence, say $\{Q_{m_j}\}_{j\in\NN}$, whose limit is a unitary matrix (see, for instance, \cite[Lemma 2.1.8]{HoJo13}). Set
$$
{\cal H}_m(\la):=\left[\begin{array}{ccc}{\cal A}_m(\la)&{\cal B}_m(\la)&{\cal S}_{\tiny\mbox{\rm right}}^{(m)}(\la)\\{\cal B}_m(\la)^*&{\cal S}_{\tiny\mbox{\rm reg}}^{(m)}(\la)&0\\{\cal S}_{\tiny\mbox{\rm right}}^{(m)}(\la)^*&0&0\end{array}\right]
$$
for the matrix in the middle of the right-hand side in \eqref{Sm}. Then the sequence $\{{\cal H}_{m_j}\}_{j\in\NN}$ is convergent as well, since ${\cal H}_{m_j}(\la)= Q_{m_j}{\cal S}_{m_j}(\la)Q_{m_j}^*$, and both $\{Q_{m_j}\}_{j\in\NN}$ (and, as a consequence, $\{Q_{m_j}\}_{j\in\NN}$ and $\{Q_{m_j}^*\}_{j\in\NN}$) and $\{{\cal S}_{m_j}\}_{j\in\NN}$ are convergent, because any subsequence of ${\cal S}_m$ converges to its limit. Moreover, by continuity of the zero blocks in the block-structure, $\{{\cal H}_{m_j}\}_{j\in\NN}$ converges to a matrix pencil of the form
\begin{equation}\label{h}
{\cal H}(\la)=\left[\begin{array}{ccc}{\cal A}(\la)&{\cal B}(\la)&{\cal C}(\la)\\{\cal B}(\la)^*&{\cal R}(\la)&0\\{\cal C}(\la)^*&0&0\end{array}\right],
\end{equation}
with
\begin{itemize}
\item ${\cal C}(\la)$ being of size $d\times (n-r+d)$,

\item ${\cal C}(\la)^*$ being of size $(n-r+d)\times d$, and

\item ${\cal R}(\la)$ being of size $(r-2d)\times(r-2d)$.
\end{itemize}

Therefore, the sequence $\{{\cal S}_{m_j}\}_{j\in\NN}$ converges to $Q^*{\cal H}(\la)Q$, where $Q:=\displaystyle\lim_{j\rightarrow\infty}Q_{m_j}$ is unitary. Since $\{{\cal S}_m\}_{m\in\NN}$ is convergent, any subsequence must converge to its limit, so $\displaystyle\lim_{m\rightarrow\infty}{\cal S}_m={\cal S}=Q^*{\cal H}(\la)Q$.  In the rest of the proof, it is important to bear in mind that ${\cal S}(\la)$ (resp. ${\cal S}(\la_0)$ for any particular $\lambda_0 \in \CC$) has rank at most $r$, and has rank exactly $r$ if and only if $\rank {\cal C}(\la) = \rank {\cal C}(\la)^* = d$ (resp. $\rank {\cal C}(\la_0) = \rank {\cal C}(\la_0)^* = d$) and $\rank {\cal R}(\la) = r-2d$ (resp. $\rank {\cal R}(\la_0) = r-2d$). This follows easily from the block structure of $\cH (\la)$ and the fact that $\rank {\cal C}(\la) \leq d$ and $\rank {\cal C}(\la)^* \leq d$.

 Let us first assume that $d'<d$ and $\lim_{m\rightarrow\infty}\cS_m=\cS\in\bun^H(\cK_{c',d'})$. Then,
\begin{eqnarray}
{\cal S}\in\bun^H({\cal K}_{c',d'}),\qquad \mbox{\rm with $d'<d$}\label{buncond}, \qquad\mbox{\rm and}\\
\mbox{\rm ${\cal S}$ has exactly $r-2d'$ (real) simple eigenvalues}\label{different}.
\end{eqnarray}
 Note that conditions \eqref{buncond}--\eqref{different} mean that ${\cal S}$ is congruent to ${\cal K}_{c',d'}$, for some real eigenvalues $a_1,\hdots,a_{r-2d'}$ in \eqref{max}, different from each other. Moreover, \eqref{buncond} implies that
 \begin{equation}\label{rankcond}
 \rank {\cal S}=r,
 \end{equation}
   which is equivalent to $\rank {\cal C}(\la) = \rank {\cal C}(\la)^* = d$ and  $\rank {\cal R}(\la) = r-2d$. Then, ${\cal R}$ in \eqref{h} is a regular pencil with $r-2d$ eigenvalues (counting multiplicities). Let us denote these eigenvalues by $\widetilde a_1,\hdots,\widetilde a_{r-2d}$ (in principle some of them might be infinite). By \eqref{buncond}, the pencil ${\cal S}$ has more than $r-2d$ eigenvalues, which are all real and distinct.
  If  $\rank {\cal C}(\la_0)=\rank {\cal C}(\la_0)^*=d$ for all $\la_0\in\RR$, then $\rank {\cal S}(\mu)=\rank {\cal S}=r$ for all $\mu\in\RR$ such that $\mu\neq \widetilde a_i$ ($i=1,\hdots,r-2d$), which means that ${\cal S}$ has  at most $r-2d$ real eigenvalues, which is a contradiction. Therefore, there must be some $\la_0\in\RR$ such that $\rank {\cal C}(\la_0) = \rank {\cal C}(\la_0)^* <d$. In particular, such $\la_0$ is an eigenvalue of ${\cal S}$, since the number of linearly independent rows of ${\cal S}(\la_0)$ is less than $r$.
  Now, we are going to see that, in this case, $\la_0$ is an eigenvalue of ${\cal S}$ with algebraic multiplicity at least 2, which is in contradiction with \eqref{different} as well. For this, we will prove that all $r\times r$ non-identically zero minors of ${\cal H}$ have $(\la-\la_0)^2$ as a factor. In order for an $r\times r$ submatrix of ${\cal H}$ to have non-identically zero determinant, it must contain fewer than $d+1$ columns among the last $n-r+d$ columns of ${\cal H}$ (namely, those corresponding to ${\cal C}$), and fewer than $d+1$ rows among the last $n-r+d$ rows of ${\cal H}$ (namely, those corresponding to ${\cal C}^*$). This is because any set of $d+1$ columns among the last $n-r+d$ columns of ${\cal H}$ is linearly dependent, and the same for the last $n-r+d$ rows. As a consequence, any $r\times r$ non-identically zero minor, $M(\la)$, of ${\cal H}$ is of the form:
 $$
 M(\la):=\det\left[\begin{array}{ccc}{\cal A}(\la)&{\cal B}(\la)&{\cal C}_1(\la)\\{\cal B}(\la)^*&{\cal R}(\la)&0\\{\cal C}_2(\la)^*&0&0\end{array}\right],
 $$
 where ${\cal C}_1$ (respectively, ${\cal C}_2^*$) is a submatrix of ${\cal C}$ (resp., ${\cal C}^*$) with size $d\times d$. Therefore, $M(\la)=\pm\det {\cal R}\cdot\det{\cal C}_1\cdot\det{\cal C}_2^*$. Since $\rank{\cal C}_1(\la_0)<d$ and $\rank{\cal C}_2(\la_0)^*<d$, the binomial $(\la-\la_0)=(\la-\overline{\la_0})$ is a factor of both $\det {\cal C}_1$ and $\det{\cal C}_2^*$, so $(\la-\la_0)^2$ is a factor of $M(\la)$.  Note that $\lambda_0\in\RR$ is key in this conclusion, in order to guarantee $(\la-\la_0)=(\la-\overline{\la_0})$.

 Therefore, the gcd of all $r\times r$ non-identically zero minors of ${\cal H}$ is a multiple of $(\la-\la_0)^2$. This implies (see, for instance, \cite[p. 141]{Gant59}) that the algebraic multiplicity of $\la_0$ as an eigenvalue of ${\cal H}$, and so of ${\cal S}$, is at least $2$.

Now, assume that $d'=d$ but $c'\neq c$.  Then, $\lim_{m\rightarrow\infty}\cS_m=\cS\in\bun^H(\cK_{c',d})$, which implies that $\cS(\la)$ is congruent to $\cK_{c',d} (\la)$ as in \eqref{max}, for some real distinct eigenvalues $a_1 , \ldots , a_{r-2d}$. Thus, if we express $\cS(\la) = \la X + Y$ and  $\cK_{c',d} (\la) = \la
X_{c',d} + Y_{c',d}$, with $X,Y, X_{c',d}$ and $Y_ {c',d}$ being constant Hermitian matrices, then $X$ and $X_ {c',d}$ are $*$-congruent and both have the same signature. This signature is
\begin{equation} \label{eq.auxxfrosig1}
 \mathrm{signature} (X) = (c'+m_+, r-2d-c'+ m_-,m_0),
\end{equation}
where $(m_+, m_-,m_0)$ is the signature of the blocks $\diag(\overbrace{{\cal M}_{\alpha+1},\hdots,{\cal M}_{\alpha+1}}^{s},
\overbrace{{\cal M}_{\alpha},\hdots,{\cal M}_{\alpha}}^{n-r-s})$ in $\cK_{c',d} (\la)$. On the other hand, if we express $\cS_m(\la) = \la X_m + Y_m \in \bun^H(\cK_{c,d})$ and  $\cK_{c,d} (\la) = \la X_{c,d} + Y_{c,d}$, with $X_m,Y_m, X_{c,d}$ and $Y_ {c,d}$ being constant Hermitian matrices, then $X_m$ and $X_ {c,d}$ are $*$-congruent and both have the same signature. This signature is
\begin{equation} \label{eq.auxxfrosig2}
 \mathrm{signature} (X_m) = (c+m_+, r-2d-c+ m_-,m_0).
\end{equation}
Thus,  $c'\ne c$ implies that $c+m_+ < c'+m_+$ or $ r-2d-c+ m_- <  r-2d-c'+ m_-$, which together with $\lim_{m\rightarrow\infty} X_m=X$, \eqref{eq.auxxfrosig1}, and \eqref{eq.auxxfrosig2} contradict Theorem 4.3 in \cite{LaRo05}.

\end{proof}

It is interesting to compare Theorem \ref{main_th} with Theorem 3 in \cite{ddd20}, which gives the generic complete eigenstructures of symmetric $n\times n$ pencils with bounded rank. In particular, the generic singular symmetric pencils contain complex eigenvalues, that may be non-real. However, this is not the case of generic Hermitian pencils, that can only contain real eigenvalues. Moreover, the number of generic eigenstructures for Hermitian pencils, provided in Theorem \ref{main_th}, is larger than the one for symmetric pencils, and this is due to the presence of the sign characteristic. The following result, which is an immediate consequence of Theorem \ref{main_th}, provides the number of different generic complete eigenstructures in the Hermitian case.

\begin{corollary}
The number of different generic bundles provided in Theorem {\rm\ref{main_th}} is equal to $\displaystyle\left(\left\lfloor\frac{r}{2}\right\rfloor+1\right)\left\lfloor\frac{r+3}{2}\right\rfloor$.
\end{corollary}
\begin{proof}
The result is a consequence of the following identities:
$$
\begin{array}{ccl}
\displaystyle\sum_{d=0}^{\lfloor r/2\rfloor}(r-2d+1)&=&\displaystyle(r+1)\sum_{d=0}^{\lfloor r/2\rfloor}1-2\sum_{d=0}^{\lfloor r/2\rfloor}d=(r+1)\left(\left\lfloor\frac{r}{2}\right\rfloor+1\right)-\left\lfloor\frac{r}{2}\right\rfloor\left(\left\lfloor\frac{r}{2}\right\rfloor+1\right)\\&=&\displaystyle\left\lfloor\frac{r+3}{2}\right\rfloor\left(\left\lfloor\frac{r}{2}\right\rfloor+1\right).
\end{array}
$$
\end{proof}

In particular, when $r=n-1$, there are $(\lfloor n/2\rfloor+1)(\lfloor(n-1)/2\rfloor+1)$ different generic bundles. This implies that the set of complex singular Hermitian $n\times n$ pencils is the union of $(\lfloor n/2\rfloor+1)(\lfloor(n-1)/2\rfloor+1)$ different bundle closures. This number is greater than $1$, provided that $n\geq2$. This is in contrast with Theorem 4 in \cite{Wate84}, where it is claimed that the set of singular complex Hermitian $n\times n$ pencils is an irreducible closed set. Then, according to this result, the union of all closures of the generic Hermitian bundles of rank $n-1$ would be an irreducible closed (in the Zariski topology) set. Related to this, it is interesting to compare with the case of symmetric pencils, considered in \cite{ddd20}. It is proved in that reference (Theorem 3) that the set of symmetric singular pencils is the union of a finite number of bundle closures, a result which is the counterpart of Theorem \ref{main_th} for symmetric pencils. Moreover, in Theorem 4 of \cite{ddd20} it is also proved that each of these bundle closures is a closed (in the Zariski topology) irreducible set. However, the arguments in the proof of that result are not valid anymore for Hermitian pencils, since the map $\Phi$ in that proof, which is essentially a congruence of matrix pencils, in the case of Hermitian pencils should be a $*$-congruence instead, which is not a polynomial map.

\begin{remark}[Skew-Hermitian matrix pencils]
A matrix pencil ${\cal N}(\la)=A+\la B$, with $A,B\in\CC^{n\times n}$, is called {\em skew Hermitian} when $A^*=-A$ and $B^*=-B$. Notably, ${\cal N}$ is skew Hermitian if and only if ${\rm i}{\cal N}$ is Hermitian.  Therefore the generic complete eigenstructures of skew-Hermitian matrix pencils can be obtained from Theorems {\rm\ref{regular_th}} and {\rm\ref{main_th}}, multiplying \eqref{max-real} and \eqref{max} by ${\rm i}$.
\end{remark}

\section{Codimension computations}\label{codim_sec}
The Hermitian orbit of a Hermitian pencil $\cH(\la)$ is a differentiable manifold over $\RR$, whose tangent space at the point $\cH$ is the following vector subspace of $\pen_{n\times n}^H$ (over $\RR$)
\begin{equation}\label{tansp}
T^H(\cH):=\{P^*\cH(\la) +\cH(\la)P :\,\ P\in \CC^{n \times n}\}.
\end{equation}
A way to see that this is the tangent space follows similar arguments to the ones used in \cite[p. 74]{DeEd95} for orbits under strict equivalence. More precisely, consider a small perturbation of $\cH$ in $\orb^H(\cH)$, namely $(I+\delta P)^*\cH(\la)(I+\delta P)=\cH(\la)+\delta\cdot \left(P^*\cH(\la)+\cH(\la)P\right) +O(\delta ^2)$, for some ``small" real quantity $\delta$ and $P\in\CC^{n\times n}$, and then take the first-order term of this perturbation (in $\delta$), namely $P^*\cH(\la)+\cH(\la)P$ (see also \cite[p. 1432]{DmKa14} for the congruence orbit). Note also that, since $\cH$ is Hermitian, then all points in $T^H(\cH)$ belong to $\pen_{n\times n}^H$, so $T^H(\cH)$ is a (real) vector subspace of $\pen_{n\times n}^H$.

The dimension of $T^H(\cH)$ is the
\emph{dimension of the Hermitian orbit} of
$\cH$ and the dimension of the normal space to the orbit at the point
$\cH$ is the \emph{codimension of the Hermitian orbit} of
$\cH$ (denoted by $\codim_{\RR}  \orb^H (\cH)$),
where the orthogonality is defined with respect to the Frobenius inner product
$\langle A + \lambda B, C + \lambda D \rangle=\tr (AC^*+BD^*),$
and $\tr (X) $ denotes the trace of the matrix $X$. We also emphasize that the normal space to $T^H(\cH)$ is considered in the vector space $\pen_{n\times n}^H$ (and not in $\pen_{n\times n}$).

In order to get the codimension of $\orb^H(\cH)$ we need the following result.

\begin{theorem}\label{codimtodim}
The codimension of the Hermitian orbit
of an $n \times n$ Hermitian matrix pencil $A + \lambda B$ is
equal to the dimension of the solution space of
\begin{equation}
\label{syst}
\begin{split}
X^*A+AX=0, \\
X^*B+BX=0.
\end{split}
\end{equation}
\end{theorem}
\begin{proof}
Define the mapping
$$f: \C^{n \times n} \rightarrow T^H(A + \lambda B), \quad X \mapsto X^* (A + \lambda B) + (A + \lambda B) X,$$
where $T^H(A + \lambda B)$ is the tangent space of $\orb^H(A+\la B)$ at
the point $A + \lambda B$.
The mapping $f$ is a surjective homomorphism of vector spaces over $\RR$. Therefore $\dim_{\RR} \C^{n \times n} = \dim_{\RR} T^H(A + \lambda B) + \dim_{\RR} V(A + \lambda B),$ where $V(A + \lambda B):=\{ X \in \C^{n \times n}:\ X^* (A + \lambda B) + (A + \lambda B) X=0 \}=\{ X \in \C^{n \times n}:\ X^* A  + AX=0=X^*B+ B X \}$.
At every point $A + \lambda B$ there is an isomorphism
$$ \pen_{n\times n}^H\ \simeq T^H(A + \lambda B) \oplus N(A + \lambda B), $$
in which $N(A + \lambda B)$ is the normal space to $T^H(A+\la B)$ at the point $A + \lambda B$ with respect to the inner product. Therefore,
\begin{align*}
\codim_{\RR} \orb^H (A + \lambda B) &= \dim_{\RR} N(A + \lambda B)= \dim_{\RR} (\pen_{n\times n}^H) - \dim_{\RR} T^H(A + \lambda B) \\
&= \dim_{\RR} (\pen_{n\times n}^H) - \dim_{\RR} \C^{n \times n} + \dim_{\RR} V(A,B) \\
&= 2n^2 - 2n^2 + \dim_{\RR} V(A,B) =\dim_{\RR} V(A,B).
\end{align*}
\end{proof}

By Theorem \ref{codimtodim}, for obtaining the codimension of $\orb^H(A+\la B)$, when $A+\la B$ is Hermitian, it is enough to obtain the dimension over $\mathbb{R}$ of the solution space of the system of matrix equations \eqref{syst}. Now we focus on computing the latter.

Let $A + \lambda B=(A_1 + \lambda B_1)\oplus (A_2 + \lambda B_2)$  be a Hermitian matrix pencil. Consider the system of matrix equations \eqref{syst} associated with $A + \lambda B$. Partitioning the unknown matrix $X$ we rewrite the system \eqref{syst} as follows
\begin{equation*}
\begin{aligned}
\begin{bmatrix}
X_{11}^*&X_{21}^*
 \\ X_{12}^*&X_{22}^*
\end{bmatrix}\begin{bmatrix}
A_1&0 \\
0 &A_2
\end{bmatrix}+\begin{bmatrix}
A_1&0 \\
0 &A_2
\end{bmatrix}\begin{bmatrix}
X_{11}&X_{12}
 \\ X_{21}&X_{22}
\end{bmatrix}=
     \begin{bmatrix}
0& 0 \\ 0& 0
\end{bmatrix}, \\
\begin{bmatrix}
X_{11}^*&X_{21}^*
 \\ X_{12}^*&X_{22}^*
\end{bmatrix}\begin{bmatrix}
B_1&0 \\
0 &B_2
\end{bmatrix}+\begin{bmatrix}
B_1&0 \\
0 &B_2
\end{bmatrix}\begin{bmatrix}
X_{11}&X_{12}
 \\ X_{21}&X_{22}
\end{bmatrix}=
     \begin{bmatrix}
0& 0 \\ 0& 0
\end{bmatrix}.
\end{aligned}
\end{equation*}
Operating in the left-hand side of the previous identities we obtain
\begin{equation}\label{decoupled}
\begin{aligned}
\begin{bmatrix}
X_{11}^*A_1 + A_1X_{11}&X_{21}^*A_2 + A_1X_{12}
 \\ X_{12}^*A_1 + A_2X_{21}&X_{22}^*A_2 + A_2X_{22}
\end{bmatrix}=
     \begin{bmatrix}
0& 0 \\ 0& 0
\end{bmatrix}, \\
\begin{bmatrix}
X_{11}^*B_1 + B_1X_{11}&X_{21}^*B_2 + B_1X_{12}
 \\ X_{12}^*B_1 + B_2X_{21}&X_{22}^*B_2 + B_2X_{22}
\end{bmatrix}=
     \begin{bmatrix}
0& 0 \\ 0& 0
\end{bmatrix},
\end{aligned}
\end{equation}
where in each equation the off-diagonal blocks are the conjugate transposed of each other.
 In the previous system of equations there are two types of equations, namely: (a) equations of the form $X^*A+AX=0$ and (b) equations of the form $YA+BZ=0$. The coefficients $A$ and $B$ are the coefficients of the pencils in the diagonal blocks of the direct sum. Then we define, for the Hermitian matrix pencils $A_i + \lambda B_i$ and $A_j + \lambda B_j$, the following systems of matrix equations:
\begin{align*}
\syst (A_i + \lambda B_i):&\qquad \qquad X^*A_i+A_iX=0,&&X^*B_i+B_iX=0;\\
\syst (A_i + \lambda B_i,\, A_j + \lambda B_j):&\qquad \qquad ZA_j+A_iY=0,&&ZB_j+B_iY=0.
\end{align*}

In the system \eqref{syst}, one can assume via a $*$-congruence and a change of variable that $A+\la B$ is given in HKCF. Then,
 we can decouple the system into a set of systems like \eqref{decoupled} (partitioned according to the number of blocks in the HKCF, which obviously may be larger than $2$) where in each system $A_1+\la B_1$ and $A_2+\la B_2$ are canonical blocks. To obtain the dimension of the solution space of \eqref{syst} it is thus enough to sum up the dimensions of the solution spaces of all systems \eqref{decoupled}.

 Following Arnold \cite[\S 5.5]{Arno71},  and taking into account that the non-real eigenvalues of Hermitian pencils appear in conjugate pairs, given a Hermitian pencil $\cH(\la)$, the {\em codimension of $\bun^H(\cH)$}  over $\mathbb{R}$, denoted by $\codim_\RR\bun^H(\cH)$, is equal to the codimension of $\orb^H(\cH)$ minus the number of different eigenvalues of $\cH$ (see also \cite[p. 1441]{DmKa14} for congruence bundles of general pencils). From this definition and the codimension of $\orb^H({\cal T}_{c,d})$ we can obtain the codimension of $\bun^H({\cal T}_{c,d})$, with ${\cal T}_{c,d}(\la)$ being the generic Hermitian pencils in either Theorem \ref{regular_th} or \ref{main_th}.

 \subsection{Codimension of generic regular bundles}
 \label{codimreg}

 In this section, we compute the codimension of the generic bundles in Theorem \ref{regular_th}.

 We start by computing the dimension of the solution space of $\syst (A_i + \lambda B_i)$ for $A_i + \lambda B_i \in \{ \sigma  \cJ_1(a), \cJ_1^H(\mu,\overline{\mu})  \}$, with $a \in \RR$ and $\mu \in \CC$ having positive imaginary part. First consider $\sigma =  1$ and $a \in \RR$, resulting in

$$
\syst (\cJ_1(a)) : \qquad  \overline{x}a+a x=0, \quad \overline{x}+ x=0,
$$
whose solution is $x = {\rm i} b$, with $b  \in \RR$. The dimension of  this solution space is equal to $1$.
The solution of $\syst (-\cJ_1(a))$ (i.e., $\sigma =  -1$) remains the same.

Consider the system $\syst(\cJ_1^H(\mu,\overline{\mu}))$:
\begin{equation} \label{JHsyst}
\begin{aligned}
\begin{bmatrix}
\overline{x_{11}}&\overline{x_{21}}
 \\ \overline{x_{12}}&\overline{x_{22}}
\end{bmatrix}
\begin{bmatrix}
0&1 \\
1&0
\end{bmatrix}+\begin{bmatrix}
0&1 \\
1&0
\end{bmatrix}\begin{bmatrix}
x_{11}&x_{12}
 \\ x_{21}&x_{22}
\end{bmatrix}=
     \begin{bmatrix}
0& 0 \\ 0& 0
\end{bmatrix}, \\
\begin{bmatrix}
\overline{x_{11}}&\overline{x_{21}}
 \\ \overline{x_{12}}&\overline{x_{22}}
\end{bmatrix}
\begin{bmatrix}
0&\overline{\mu} \\
\mu&0
\end{bmatrix}+\begin{bmatrix}
0& \overline{\mu} \\
\mu&0
\end{bmatrix}\begin{bmatrix}
x_{11}&x_{12}
 \\ x_{21}&x_{22}
\end{bmatrix}=
     \begin{bmatrix}
0& 0 \\ 0& 0
\end{bmatrix}.
\end{aligned}
\end{equation}
Multiplying the matrices in the first equation in \eqref{JHsyst} we have
\begin{equation*}
\begin{aligned}
\begin{bmatrix}
\overline{x_{21}}+x_{21}&
\overline{x_{11}}+x_{22}\\
\overline{x_{22}}+x_{11}&
\overline{x_{12}}+x_{12}
\end{bmatrix} =
\begin{bmatrix}
0& 0 \\ 0& 0
\end{bmatrix},
\end{aligned}
\end{equation*}
whose solution is then replaced into the second equation:
\begin{equation*}
\begin{aligned}
\begin{bmatrix}
\overline{x_{11}}& -{\rm i} b_{21}
 \\ -{\rm i} b_{12}&-{x}_{11}
\end{bmatrix}
\begin{bmatrix}
0&\overline{\mu} \\
\mu&0
\end{bmatrix}+\begin{bmatrix}
0& \overline{\mu} \\
\mu&0
\end{bmatrix}\begin{bmatrix}
x_{11}& {\rm i} b_{12}\\
{\rm i} b_{21}&-\overline{x_{11}}
\end{bmatrix}=
     \begin{bmatrix}
0& 0 \\ 0& 0
\end{bmatrix},
\end{aligned}
\end{equation*}
where $b_{12}, b_{21} \in \mathbb{R}$, or, equivalently,
\begin{equation*}
\begin{aligned}
\begin{bmatrix}
{\rm i} b_{21} (\overline{\mu} - \mu) & \overline{x_{11}} \overline{\mu} - \overline{x_{11}} \overline{\mu}\\
 {x}_{11}\mu -{x}_{11}\mu & {\rm i} b_{12} (\mu -\overline{\mu})
\end{bmatrix}=
\begin{bmatrix}
0& 0 \\ 0& 0
\end{bmatrix}.
\end{aligned}
\end{equation*}
Therefore $b_{12} = b_{21}=0 , x_{11} \in \CC$, and the dimension over $\RR$ of the solution space of $\syst(\cJ_1^H(\mu,\overline{\mu}))$ is $2$.

Now we compute the dimension of the solution space of $\syst (A_i + \lambda B_i,\, A_j + \lambda B_j)$ for $A_i + \lambda B_i, A_j + \lambda B_j \in \{ \sigma  \cJ_1(a), \cJ_1^H(\mu,\overline{\mu}) \}$, with $a \in  \RR, \mu \in\CC$ and $\im \mu >0$.

The system $\syst (\sigma_i  \cJ_1(a_i),\, \sigma_j  \cJ_1(a_j))$ reads $z+y=0=a_jz+a_iy$ when $\sigma_i=\sigma_j$ and $z-y=0=a_jz-a_iy$ when $\sigma_i=-\sigma_j$. Since $a_i\neq a_j$, in both cases the only solution is $y=z=0$, so the dimension of the solution space of $\syst (\sigma_i  \cJ_1(a_i),\, \sigma_j  \cJ_1(a_j))$ is $0$.

Next type of systems we need to consider is $\syst(\cJ_1^H(\mu,\overline{\mu}),\cJ_1^H(\eta,\overline{\eta}))$, where $\mu,  \eta$ have positive imaginary parts and $\mu \neq \eta$:
\begin{equation*}
\begin{aligned}
\begin{bmatrix}
z_{11}&z_{12}
 \\ z_{21}&z_{22}
\end{bmatrix}
\begin{bmatrix}
0&1 \\
1&0
\end{bmatrix}+\begin{bmatrix}
0&1 \\
1&0
\end{bmatrix}\begin{bmatrix}
y_{11}&y_{12}
 \\ y_{21}&y_{22}
\end{bmatrix}=
     \begin{bmatrix}
0& 0 \\ 0& 0
\end{bmatrix}, \\
\begin{bmatrix}
z_{11}&z_{12}
 \\ z_{21}&z_{22}
\end{bmatrix}
\begin{bmatrix}
0&\overline{\mu} \\
\mu&0
\end{bmatrix}+\begin{bmatrix}
0& \overline{\eta} \\
\eta&0
\end{bmatrix}\begin{bmatrix}
y_{11}&y_{12}
 \\ y_{21}&y_{22}
\end{bmatrix}=
     \begin{bmatrix}
0& 0 \\ 0& 0
\end{bmatrix},
\end{aligned}
\end{equation*}
or, equivalently,
\begin{equation*}
\begin{aligned}
\begin{bmatrix}
z_{11}&z_{12}
 \\ z_{21}&z_{22}
\end{bmatrix}
\begin{bmatrix}
0&\overline{\mu} \\
\mu&0
\end{bmatrix}-\begin{bmatrix}
0& \overline{\eta} \\
\eta&0
\end{bmatrix}\begin{bmatrix}
z_{22}&z_{21}
 \\ z_{12}&z_{11}
\end{bmatrix}=
     \begin{bmatrix}
0& 0 \\ 0& 0
\end{bmatrix}.
\end{aligned}
\end{equation*}
After performing the matrix multiplications we obtain:
\begin{equation*}
\begin{aligned}
\begin{bmatrix}
z_{12} (\mu - \overline{\eta}) & z_{11} (\overline{\mu} - \overline{\eta}) \\
z_{22} (\mu -\eta) & z_{21} (\overline{\mu} - \eta)
\end{bmatrix}=
     \begin{bmatrix}
0& 0 \\ 0& 0
\end{bmatrix}.
\end{aligned}
\end{equation*}
If ${\rm re}(\mu) \neq  {\rm re}(\eta)$ then $z_{ij} = 0$. If ${\rm re}(\mu)  =  {\rm re}(\eta)$, taking into   account that ${\rm im}(\mu), {\rm im}(\eta) > 0$, and $\mu \neq \eta$, we have ${\rm im}(\mu) \neq  \pm {\rm im}(\eta)$, which in turn result in $z_{ij} = 0$.  Thus, the dimension of the solution space of $\syst(\cJ_1^H(\mu,\overline{\mu}),\cJ_1^H(\eta,\overline{\eta}))$ is zero.

Finally, we consider the systems $\syst(\cJ_1^H(\mu,\overline{\mu}), \sigma \cJ_1(a))$, where $\mu$ has a positive imaginary part, $\sigma = \pm 1$, and $a \in  \RR$. For $\sigma = 1$, the system is
\begin{equation*}
\begin{aligned}
\begin{bmatrix}
z_{11}&z_{12}
\end{bmatrix}
\begin{bmatrix}
0&1 \\
1&0
\end{bmatrix}+\begin{bmatrix}
1
\end{bmatrix}\begin{bmatrix}
y_{11}&y_{12}\\
\end{bmatrix}=
     \begin{bmatrix}
0& 0
\end{bmatrix}, \\
\begin{bmatrix}
z_{11}&z_{12}
\end{bmatrix}
\begin{bmatrix}
0&\overline{\mu} \\
\mu&0
\end{bmatrix}+
\begin{bmatrix}
a
\end{bmatrix}\begin{bmatrix}
y_{11}&y_{12}
\end{bmatrix}=
     \begin{bmatrix}
0& 0
\end{bmatrix},
\end{aligned}
\end{equation*}
which is equivalent to a single equation
\begin{equation*}
\begin{aligned}
\begin{bmatrix}
z_{11}&z_{12}
\end{bmatrix}
\begin{bmatrix}
0&\overline{\mu} \\
\mu&0
\end{bmatrix}+
\begin{bmatrix}
a
\end{bmatrix}\begin{bmatrix}
-z_{12}&-z_{11}
\end{bmatrix}=
     \begin{bmatrix}
0& 0
\end{bmatrix}.
\end{aligned}
\end{equation*}
Multiplying the matrices we get
\begin{equation*}
\begin{aligned}
\begin{bmatrix}
z_{12} \mu - a z_{12} & z_{11} \overline{\mu} - a z_{11}
\end{bmatrix}=
     \begin{bmatrix}
0& 0
\end{bmatrix}.
\end{aligned}
\end{equation*}
Since ${\rm im}(\mu) > 0$, we must have $z_{ij} = 0$ and thus $y_{ij} = 0$. For $\sigma = -1$, the solution is exactly the same. Thus, the dimension of the solution space of  $\syst(\cJ_1^H(\mu,\overline{\mu}), \sigma \cJ_1(a))$ is always $0$.

We note that only the dimensions of the solution spaces for $\syst(\cJ_1^H(\mu,\overline{\mu}))$ and $\syst(\sigma \cJ_1(a))$  (2 and 1, respectively) contribute to the dimension of the solution spaces for $A + \lambda  B$ being equal to the generic Hermitian canonical form ${\cal R}_{c,d}$. Summing up all these dimensions of the solution spaces of matrix equations we obtain $\codim_\RR \orb^H({\cal R}_{c,d}) = n$, which implies the following theorem:
\begin{theorem} The codimension in $\pen_{n\times n}^H$ of the Hermitian bundle of ${\cal R}_{c,d} (\la)$ of generic Hermitian pencils in Theorem {\rm\ref{regular_th}} is
\begin{equation*}
\codim_\RR \bun^H({\cal R}_{c,d}) = 0.
\end{equation*}
\end{theorem}

In contrast with the singular case, considered in Section \ref{sec62} below, $\codim_\RR \orb^H({\cal R}_{c,d})$ and $\codim_\RR \bun^H({\cal R}_{c,d})$ do not depend on the values of $c$ or $d$.

 \subsection{Codimension of generic singular bundles with bounded rank}
 \label{sec62}
 In this subsection  we obtain the codimensions of the generic bundles $\bun^H(\cK_{c,d})$ in Theorem \ref{main_th}.
 We want to emphasize that, as a consequence of Theorem \ref{codim_th}, the generic bundles in Theorem \ref{main_th} have different codimension whenever $d\neq d'$ (but those for which $d=d'$ and $c\neq c'$ have the same codimension). Thus, the codimension of the generic bundles does not depend on the sign characteristic.

\begin{theorem}\label{codim_th}
The codimension in $\pen_{n\times n}^H$ of the Hermitian $n\times n$ bundle of generic Hermitian matrix pencils in Theorem~{\rm\ref{main_th}} is
\begin{equation*}
\begin{aligned}
\codim_\RR \bun^H({\cal K}_{c,d}) &= 2(n-d)(n-r).
\end{aligned}
\end{equation*}
\end{theorem}
\begin{proof}
 We are first going to compute the codimension of $\orb^H({\cal K}_{c,d})$. By Theorem~\ref{codimtodim}, this codimension is equal to the dimension of the solution space of \eqref{syst}, with ${\cal K}_{c,d}(\la)=A+\la B$.
By the arguments after the proof of Theorem~\ref{codimtodim}, we need to obtain the dimension of the solution spaces of $\syst(A_i+\la B_i)$ and $\syst(A_i+\la B_i,A_j+\la B_j)$, where $A_i+\la B_i$ and $A_j+\la B_j$ are the canonical blocks appearing in ${\cal K}_{c,d}(\la)$.

We recall that the dimension of  the solution space of $\syst (\sigma \cJ_1(a))$ for $a\in \RR$ is computed in Section \ref{codimreg} and is equal to $1$.

Now we consider the system $\syst(\cM_k)$:
\begin{equation*}
\begin{aligned}
\begin{bmatrix}
X_{11}^*&X_{21}^*
 \\ X_{12}^*&X_{22}^*
\end{bmatrix}\begin{bmatrix}
0&F_k^\top \\
F_k&0
\end{bmatrix}+\begin{bmatrix}
0&F_k^\top \\
F_k&0
\end{bmatrix}\begin{bmatrix}
X_{11}&X_{12}
 \\ X_{21}&X_{22}
\end{bmatrix}=
     \begin{bmatrix}
0& 0 \\ 0& 0
\end{bmatrix}, \\
\begin{bmatrix}
X_{11}^*&X_{21}^*
 \\ X_{12}^*&X_{22}^*
\end{bmatrix}\begin{bmatrix}
0&G_k^\top \\
G_k&0
\end{bmatrix}+\begin{bmatrix}
0&G_k^\top \\
G_k&0
\end{bmatrix}\begin{bmatrix}
X_{11}&X_{12}
 \\ X_{21}&X_{22}
\end{bmatrix}=
     \begin{bmatrix}
0& 0 \\ 0& 0
\end{bmatrix},
\end{aligned}
\end{equation*}
where $X$ is partitioned conformally with the $2 \times 2$ block structure of $\cM_k.$ Note that the conjugation of $X$ is the only difference compared to the case described in \cite[Section 3.2]{DmKS14}. Multiplying the matrices we have
\begin{equation} \label{lnsyst}
\begin{aligned}
\begin{bmatrix}
X_{21}^*F_k+F_k^\top X_{21}&
X_{11}^*F_k^\top+F_k^\top X_{22}\\
X_{22}^*F_k+F_kX_{11}&
X_{12}^*F_k^\top+F_kX_{12}
\end{bmatrix} =
\begin{bmatrix}
0& 0 \\ 0& 0
\end{bmatrix}, \\
\begin{bmatrix}
X_{21}^*G_k+G_k^\top X_{21}&
X_{11}^*G_k^\top+G_k^\top X_{22}\\
X_{22}^*G_k+G_kX_{11}&
X_{12}^*G_k^\top+G_kX_{12}
\end{bmatrix} =
\begin{bmatrix}
0& 0 \\ 0& 0
\end{bmatrix}.
\end{aligned}
\end{equation}
Since the pairs of blocks at positions $(1,2)$ and $(2,1)$ are equal to each other up to transposition and conjugation, equation \eqref{lnsyst} decomposes into three independent subsystems.
Let us first consider the subsystem corresponding to the $(1,1)$-blocks:
\begin{equation} \label{systl}
\begin{aligned}
X_{21}^*F_k+F_k^\top X_{21}=0, \\
X_{21}^*G_k+G_k^\top X_{21}=0.
\end{aligned}
\end{equation}
In order to satisfy the second equation of \eqref{systl}, $X_{21}$ must have the form
$$X_{21}=\begin{bmatrix}
{\rm i} b_{11}&x_{12}&x_{13}&\ldots & x_{1k}&0 \\
-\overline{x_{12}}&{\rm i} b_{22}&x_{23}&\ldots & x_{2k}&0 \\
-\overline{x_{13}}&-\overline{x_{23}}&{\rm i} b_{33}&\ldots & x_{3k}&0 \\
\vdots&\vdots&\vdots&\ddots & \vdots & \vdots \\
-\overline{x_{1k}}&-\overline{x_{2k}}&-\overline{x_{3k}}&\ldots & {\rm i} b_{kk}&0 \\
\end{bmatrix},$$
 with $b_{11},\hdots,b_{kk}\in\RR$.
Substituting $X_{21}$ in the first equation of \eqref{systl},
we obtain
\begin{equation}\label{systll}
 \begin{bmatrix}
0&  - {\rm i} b_{11}& -x_{12}& \ldots &  -x_{1,k-1}& -x_{1k} \\
{\rm i} b_{11}& \overline{x_{12}} + x_{12} &-{\rm i} b_{22}- x_{13}&\ldots & x_{1k}-x_{2,k-1}& -x_{2k}\\
-\overline{x_{12}}& {\rm i} b_{22}+\overline{x_{13}}& \overline{x_{23}} + x_{23} &\ldots& x_{2k}-x_{3,k-1}& -x_{3k}\\
\vdots&\vdots&\vdots&\ddots&\vdots&\vdots\\
-\overline{x_{1,k-1}}& \overline{x_{1k}}-\overline{x_{2,k-1}}& \overline{x_{2k}}-\overline{x_{3,k-1}}&\ldots & \overline{x_{k-1,k}} + x_{k-1,k} & -{\rm i} b_{kk}\\
-\overline{x_{1k}} & -\overline{x_{2k}}&-\overline{x_{3k}} &\ldots & {\rm i} b_{kk}& 0
\end{bmatrix}=0.
\end{equation}
 The first row of \eqref{systll} gives $b_{11}=0=x_{12}=\cdots=x_{1k}$. Replacing this into the equations obtained from the upper diagonal entries in the second row of \eqref{systll} gives $b_{22}=0=x_{23}=\cdots=x_{2k}$. Proceeding recursively in this way, we conclude that $X_{21}=0$.

Now consider the subsystem corresponding to the $(2,2)$-blocks:
\begin{equation} \label{systl1}
\begin{aligned}
X_{12}^*F_k^\top+F_kX_{12}=0, \\
X_{12}^*G_k^\top+G_kX_{12}=0.
\end{aligned}
\end{equation}
In order to satisfy the second equation of \eqref{systl1}, $X_{12}^*$ must have the form
$$X_{12}^*=\begin{bmatrix}
 {\rm i} b_{11}&x_{12}&x_{13}&\ldots & x_{1k}& x_{1,k+1}\\
-\overline{x_{12}}& {\rm i} b_{22}&x_{23}&\ldots & x_{2k}& x_{2,k+1}\\
-\overline{x_{13}}&-\overline{x_{23}}& {\rm i} b_{33}&\ldots & x_{3k}& x_{3,k+1}\\
\vdots&\vdots&\vdots&\ddots & \vdots & \vdots\\
-\overline{x_{1k}}&-\overline{x_{2k}}&-\overline{x_{3k}}&\ldots &  {\rm i} b_{kk}& x_{k,k+1}\\
\end{bmatrix},$$
with $b_{11},\hdots,b_{kk}\in\RR$.
Replacing it in the first equation of \eqref{systl1}, we obtain
\begin{equation*}
 \begin{bmatrix}
\overline{x_{12}} + x_{12} & x_{13} - {\rm i} b_{22}  &x_{14}-x_{23}&\ldots & x_{1,k+1}-x_{2k}\\
\overline{x_{13}} + {\rm i} b_{22} & x_{23}  + \overline{x_{23}} &x_{24}-{\rm i}b_{33} &\ldots & x_{2,k+1}-x_{3k}\\
\vdots&\ddots&\ddots & \ddots & \vdots\\
 \overline{x_{1k}}-\overline{x_{2,k-1}}& \overline{x_{2k}}-\overline{x_{3,k-1}}&\ldots& x_{k-1,k} + \overline{x_{k-1,k}}& x_{k-1,k+1} - {\rm i} b_{kk} \\
 \overline{x_{1,k+1}}-\overline{x_{2k}}&\overline{x_{2,k+1}}-\overline{x_{3k}} &\ldots & \overline{x_{k-1,k+1}}+ {\rm i} b_{kk} & x_{k,k+1}  + \overline{x_{k,k+1}}
\end{bmatrix}=0.
\end{equation*}
Equating the diagonal entries of this identity gives $x_{i,i+1}+\overline{x_{i,i+1}}=0$, for $i=1,\hdots,k$, which implies that $x_{i,i+1}={\rm i}b_{i,i+1}$, with $b_{i,i+1}\in\RR$. Equating the upper diagonal entries in turn gives $x_{i,j}=x_{i+1,j-1}$, for $j\geq i+2$.
Therefore, $X_{12}^* =  {\rm i} [b_{ij}]$ is a Hankel matrix (namely, each skew diagonal is constant) and $b_{ij} \in \RR$. 

 Finally, using similar arguments to the ones in \cite[Section 3.2]{DmKS13} (replacing $\top$ by $*$ in that reference leads to the same solution), the solution of the off-diagonal subsystem
\begin{equation*}
\begin{aligned}
X_{22}^*F_k+F_kX_{11}=0, \\
X_{22}^*G_k+G_kX_{11}=0,
\end{aligned}
\end{equation*}
is  $X_{11}= \alpha I_{k+1}$ and $X_{22}= - \overline{\alpha} I_{k}$,
with $\alpha \in \CC$.

Summing up, the solution of system \eqref{lnsyst} is
$$X=\begin{bmatrix}
 \alpha I_{k+1}&X_{12}\\
0_{k,k+1}& - \overline{\alpha} I_{k}
\end{bmatrix},$$
 where, as we have seen, $X_{12}^* =  {\rm i} [b_{ij}]$ is a Hankel matrix and $b_{ij} \in \RR$. The number of independent real parameters of the matrix $X$ above is $2k+2$, namely $2k$ coming from $X_{12}$ and $2$ coming from $\alpha\in\CC$. Hence the dimension over $\RR$ of the solution space of $\syst(\cM_k)$ is $2k+2$.

Now we compute the dimension of the solution space of $\syst (A_i + \lambda B_i,\, A_j + \lambda B_j)$ for $A_i + \lambda B_i, A_j + \lambda B_j \in \{ \sigma  \cJ_1(a), \cM_k  \}$, with $a \in \RR$.
The system $\syst (\sigma_i  \cJ_1(a_i ),\, \ \sigma_j  \cJ_1(a_j))$ was handled in Section \ref{codimreg} and the dimension of its solution space is equal to $0$. The dimension of the solution of $\syst (\sigma\cJ_1(a),\, \cM_k)$, and $\syst (\cM_{m_i},\, \cM_{m_j})$ follows directly from the dimension of the corresponding systems in \cite[Corollary 2.2]{DmKS14}, see also \cite[Corollary 2.1]{Dmyt19} and \cite[Theorem 2.7]{DmJK13}. Namely, the dimension of
$\syst (\cJ_1(a),\, \cM_k)$ is equal to $2$ (from here, we can conclude that the dimension of the solution space of $\syst (- \cJ_1(a),\, \cM_k)$ is also equal to $2$, using the change of variables $Y'=-Y$); and as for $\syst (\cM_{m_i},\, \cM_{m_j})$ the dimension is equal to $2\cdot(2 \max \{ m_i,m_j \} + \varepsilon_{ij})$, where $\varepsilon_{ij} = 2$ if $m_i=m_j$ and $\varepsilon_{ij} = 1$ otherwise. Note that, for the generic Hermitian pencils in Theorem \ref{main_th}, $m_i,m_j\in\{\alpha,\alpha+1\}$.

 We summarize in Table \ref{syst_table} the dimension of the solution space of $\syst(\cH)$ and in Table \ref{systpair_table} the dimension of the solution space of the systems $\syst(\cH_1,\cH_2)$, with $\cH$, $\cH_1$, and $\cH_2$ being all possible pairs of blocks in ${\cal K}_{c,d}(\la)$. Each entry of the table contains the dimension of the solution spaces of all systems obtained from the corresponding blocks, so it is the product of the dimension obtained with the precedent arguments for each system multiplied by the number of blocks of each kind in ${\cal K}_{c,d}(\la)$. The lower diagonal entries in Table \ref{systpair_table} are not considered to avoid repetitions.

\begin{table}[]
    \centering
    \begin{tabular}{c|c}
         $\cH$&$ \sum \dim(\syst(\cH))$  \\\hline\hline
        $\sigma\cJ_1(a)$ &$r-2d$\\
        $\cM_\alpha$&$(2\alpha+2)(n-r-s)$\\
        $\cM_{\alpha+1}$&$(2\alpha+4)s$
    \end{tabular}
    \caption{Sum of dimensions of the solution spaces of $\syst(\cH)$, for all blocks of each of the forms $\cH$ in ${\cal K}_{c,d}(\la)$}
    \label{syst_table}
\end{table}
\begin{table}[]
    \centering
    \begin{tabular}{c|ccc}
         \diagbox[]{$\cH_1$}{$\cH_2$}& $\sigma\cJ_1(a_i)$&$\cM_\alpha$&$\cM_{\alpha+1}$ \\\hline
        $\sigma\cJ_1(a_i)$ &$0$&$2(r-2d)(n-r-s)$&$2(r-2d)s$\\$\cM_\alpha$&--&$2(2\alpha+2)\binom{n-r-s}{2}$&$2(2\alpha +3)(n-r-s)s$\\
        $\cM_{\alpha+1}$&--&--&$2(2\alpha+4)\binom{s}{2}$
    \end{tabular}
    \caption{Sum of dimensions of the solution spaces of $\syst(\cH_1,\cH_2)$, for all blocks of each of the forms $\cH_1,\cH_2$ in ${\cal K}_{c,d}(\la)$}
    \label{systpair_table}
\end{table}


Summing up the dimensions of the solution spaces for all the subsystems we obtain the following codimension of $\orb^H({\cal K}_{c,d})$:

\begin{equation*} \label{codimcomporb}
\begin{aligned}
\codim_\RR \orb^H({\cal K}_{c,d}) &\phantom{a} = r-2d+2(\alpha +2)s + 2(\alpha +1)(n-r-s) + 2(n-r)(r-2d)  \\ 
&\phantom{a}  + 2(n-r-s)(n-r-s-1)(2\alpha +2)/2 + 2s(s-1)(2\alpha+4)/2 \\
&\phantom{a}  + 2s(n-r-s)(2(\alpha+1) +1)\\
&\phantom{a} = r-2d+2(\alpha(n-r) +n -r+s) + 2(n-r)(r-2d)\\
&\phantom{a} + 2(n-r-s)(n-r-s-1)(\alpha +1) + 2s(s-1)(\alpha+1)  + 2s(s-1) \\
&\phantom{a} + 4s(n-r-s)(\alpha+1) + 2s(n-r-s) \\
&\phantom{a} = r-2d + 2(d + n - r) + 2(n-r)(r-2d)\\
&\phantom{a} + 2(\alpha +1) ((n-r-s)(n-r-1) + s(n-r-1)) + 2s(n-r-1) \\
&\phantom{a} = r-2d + 2(d + n - r) + 2(n-r)(r-2d) \\
&\phantom{a} + 2(\alpha +1) (n-r)(n-r-1) + 2s(n-r-1) \\
&\phantom{a} = r-2d + 2(d + n - r) + 2(n-r)(r-2d) + 2(d+n-r)(n-r-1)  \\
&\phantom{a} = r - 2d + 2(n-d)(n-r),
\end{aligned}
\end{equation*}
 where in the third and fifth identities we have used that $\alpha(n-r)+s=d$.
As a consequence, the codimensions of the generic bundles are:
\begin{equation*}
\begin{aligned}
\codim_\RR \bun^H({\cal K}_{c,d}) &= r - 2d + 2(n-d)(n-r) - r + 2d = 2(n-d)(n-r).
\end{aligned}
\end{equation*}
\end{proof}

We want to emphasize that the generic bundle with smallest codimension (namely, with largest dimension) is the one with largest $d$, namely the one having the smallest number of eigenvalues (this number is equal to $0$ if $r$ is even and to $1$ if $r$ is odd).

\section{Numerical illustration of the theoretical results}\label{experiments_sec}

Despite this paper is of a theoretical nature, we provide a couple of numerical experiments to illustrate and support the main results (namely, Theorems \ref{regular_th} and \ref{main_th}). The {\sc Matlab} code for these experiments is available on GitHub.\footnote{\href{https://github.com/dmand/generic_herm_experiments.git}{https://github.com/dmand/generic\_herm\_experiments.git}}

\begin{example}\label{ex:reggen}
The purpose of this experiment is to show that all the generic complete eigenstructures of
regular Hermitian matrix pencils in Theorem {\rm\ref{regular_th}} arise numerically when computing the eigenvalues of a simple family of randomly generated regular Hermitian matrix pencils.

Using {\rm\cite{Marc22}} we generate the matrix coefficients, i.e., $A$ and $B$, of a Hermitian matrix pencil $A + \lambda B$, and shift these matrix coefficients, by adding to each of the matrices a diagonal matrix with  the same value on the diagonal: $A + w_j I$ and $B + w_j I$. Then we compute the eigenvalues of $(A + w_jI) + \lambda (B+ w_jI)$ with the {\sc Matlab} function {\rm eig($A$,$B$)} and we count the number of real eigenvalues in the output.  
In Figure {\rm\ref{fig1}} we show the outcome after repeating this computation $350$ times for $20 \times 20$ Hermitian matrix pencils with the shifts $w_j = (j \log j)/100$. We see that the number of the real eigenvalues varies from $0$ to $20$ (the size of the pencils), and that all possible numbers (namely, all even numbers between $0$ and $20$) are attained. However, as $j$ increases, there is a larger number of real eigenvalues. This is the expected output, since the diagonal entries of the coefficients matrices are increasing, while the size of all non-diagonal entries remains the same.

\begin{figure}[h!]
\label{fig1}
  \centering
      \includegraphics[width=0.7\textwidth]{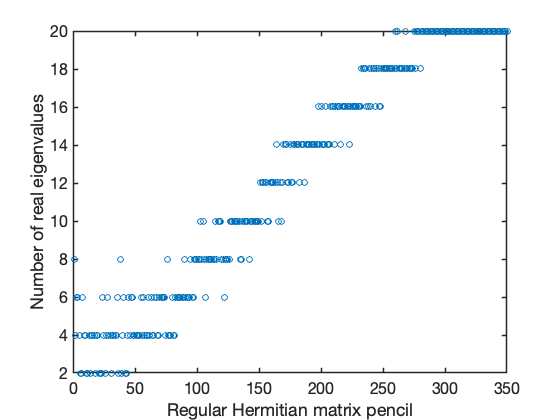}
  \caption{Number of real eigenvalues for $20 \times 20$ regular random Hermitian matrix pencils $(A + w_jI) + \lambda (B+ w_jI)$, where $w_j = (j \log j)/100, \  j =1, \dots,  350$.}
\end{figure}

\end{example}

\begin{example}\label{ex:reggen2}
The purpose of this experiment is to show that singular Hermitian matrix pencils generically do not have pairs of complex conjugate eigenvalues. We generate Hermitian matrix pencils of a given rank $r$ using the result of {\rm\cite[Theorem~2]{dmm22}} and compute their eigenvalues using the {\rm MultiParEig} Toolbox for {\sc Matlab} {\rm\cite{Ples22}}, see also {\rm\cite{HoMP19}}. In extensive set of experiments we have never seen a pencil with a pair of complex conjugate eigenvalues. For example, after running $50 000$ experiments with $17 \times17$ Hermitian matrix pencils of rank $9$, we get $9,7,5,3$, or $1$ eigenvalues (all real) and no non-real eigenvalues.
\end{example}

\section{Conclusions}\label{conclusion_sec}
 We have proved that the set of complex Hermitian $n\times n$ matrix pencils with rank at most $r$ (with $r\leq n$) is the union of a finite number of closed sets, which are the closures of the bundles of certain pencils. These pencils are given explicitly in Hermitian Kronecker canonical form, namely explicitly displaying their complete eigenstructure. Hence, these are the generic complete eigenstrustures of Hermitian $n\times n$ matrix pencils with rank at most $r$, and the corresponding bundles are the generic bundles. The case when $r=n$ is addressed separately, because it provides the generic eigenstructures of general $n\times n$ Hermitian pencils (without any rank constraint). In this case, all except one of the generic eigenstructures contain non-real eigenvalues. However, this is not the case when $r<n$, where the generic eigenstructures can only contain real eigenvalues (if any).

  We have provided the number of generic bundles, which is larger than $1$. Finally, we have obtained the (co)dimension of the generic bundles.

\bigskip

\noindent{\bf Acknowledgements.} The work of F. De Ter\'an and F. M. Dopico has been partially supported by the Agencia Estatal de Investigaci\'on of Spain through grant PID2019-106362GB-I00 MCIN/AEI/ 10.13039/501100011033/ and by the Madrid Government (Comunidad de Madrid-Spain) under the Multiannual Agreement with UC3M in the line of Excellence of University Professors (EPUC3M23), and in the context of the V PRICIT (Regional Programme of Research and Technological Innovation). The work of A. Dmytryshyn has been supported by the Swedish Research Council (VR) under grant 2021-05393.

\bibliographystyle{siamplain}

\end{document}